\documentclass[10pt,a4paper]{article}
\usepackage{fullpage,amssymb,amsmath,amsfonts,amsthm,paralist,hyperref,graphicx,color,float,array}
\usepackage{here} 
\usepackage[normalem]{ulem}

\newtheorem{lemma}{Lemma}
\newtheorem{theorem}{Theorem}

\newtheorem{observation}{Observation}
\newtheorem{problem}{Problem}
\newenvironment{facts}{\noindent\textit{Related facts. }}{\par\medskip}

\newcommand{\XSays}[3]{{\color{#2}
      {$\rule[-0.12cm]{0.2in}{0.5cm}$\fbox{\tt
            #1:} }%
      \itshape #3
      \marginpar{\color{#2}\tt #1}%
      \def\comment{#3}\def\empty{}\ifx\comment\empty\else
      {$\rule[0.1cm]{0.3in}{0.1cm}$\fbox{\tt
            end}$\rule[0.1cm]{0.3in}{0.1cm}$} \fi
   }%
}

\makeatletter
\newcommand\floatc@mybox[2]{\vbox{\hbadness10000
\moveleft3.4pt\vbox{\advance\hsize by6.8pt
\hrule \hbox to\hsize{\vrule\kern3pt
\vbox{\kern3pt\vbox{\advance\hsize by-6.8pt{\@fs@cfont #1} #2}\kern3pt}\kern3pt\vrule}}}}%
\newcommand\fs@mybox{\def\@fs@cfont{\bfseries}\let\@fs@capt\floatc@mybox
\def\@fs@pre{\setbox\@currbox\vbox{\hbadness10000
\moveleft3.4pt\vbox{\advance\hsize by6.8pt
\hrule \hbox to\hsize{\vrule\kern3pt
\vbox{\kern4.5pt\box\@currbox\kern4.5pt}\kern3pt\vrule}\hrule}}}%
\def\@fs@mid{}%
\def\@fs@post{}%
\let\@fs@iftopcapt\iftrue}
\makeatother
\floatstyle{mybox}
\newfloat{inset}{tb}{ins}
\floatname{inset}{Inset}

\title{Reptilings and space-filling curves for acute triangles}
\author{Marinus Gottschau\\Ludwig-Maximilians-Universit\"at M\"unchen\and
        Herman Haverkort\\Technische Universiteit Eindhoven\and
        Kilian Matzke\\Ludwig-Maximilians-Universit\"at M\"unchen}

\begin{document}
\maketitle

\begin{abstract}
An \emph{$r$-gentiling} is a dissection of a shape into $r \geq 2$ parts which are all similar to the original shape. An \emph{$r$-reptiling} is an $r$-gentiling of which all parts are mutually congruent. By applying gentilings recursively, together with a rule that defines an order on the parts, one may obtain an order in which to traverse all points within the original shape. We say such a traversal is a \emph{face-continuous space-filling curve} if, at any level of recursion, the interior of the union of any set of consecutive parts is connected---that is, consecutive parts must always meet along an edge. Most famously, the isosceles right triangle admits a 2-reptiling, which forms the basis of the face-continuous Sierpi\'nski space-filling curve; many other \emph{right} triangles admit reptilings and gentilings that yield face-continuous space-filling curves as well. In this study we investigate what \emph{acute} triangles admit non-trivial reptilings and gentilings, and whether these can form the basis for face-continuous space-filling curves.
We derive several properties of reptilings and gentilings of acute (sometimes also obtuse) triangles, leading to the following conclusion: no face-continuous space-filling curve can be constructed on the basis of reptilings of acute triangles.
\end{abstract}

\section{Introduction}

\subsection{Reptilings, gentilings and space-filling curves}

\paragraph{Reptilings and gentilings}
We call a geometric figure, that is, a set of points in Euclidean space, $T$ an \emph{$r$-gentile} if $T$ admits an \emph{$r$-gentiling}, that is, a subdivision of $T$ into $r \geq 2$ figures (\emph{tiles}) $T_1,...,T_r$, such that each of the figures $T_1,...,T_r$ is similar to $T$. In other words, $T$ is a $r$-gentile if we can tile it with $r$ smaller copies of itself. This generalizes the concept of \emph{reptiles}, coined by Golomb~\cite{golomb}: a figure $T$ is an \emph{$r$-reptile} if $T$ admits an \emph{$r$-reptiling}, that is, a subdivision of $T$ into $r \geq 2$ figures $T_1,...,T_r$, such that each of the figures $T_1,...,T_r$ is similar to $T$ \emph{and all figures $T_1,...,T_r$ are mutually congruent}. In other words, $T$ is a $r$-reptile if we can tile it with $r$ equally large smaller copies of itself.

\paragraph{Traversal orders}
If we augment the definitions of gentilings with an order on the tiles, we can use gentilings to define a mapping $f$ from the unit interval $[0,1)$ to a gentile $T$, such that the closure of the image of $f$ is all of $T$ and such that $f$ is measure-preserving (up to uniform scaling). In other words, the image of an interval $[a,b) \subseteq [0,1)$ is a set $T' \subseteq T$ such that $|T'| = (b-a)|T|$, where we use $|X|$ to denote the Lebesgue measure (area) of $X$. We call such a mapping a \emph{traversal order}. An example is illustrated in Figure~\ref{fig:sierpinski}. The figure shows how, by simultaneously subdividing intervals and triangular tiles, while maintaining a one-to-one correspondence between intervals and tiles, we can get an arbitrarily precise approximation of a mapping $f$ from $[0,1)$ to a triangle.

\begin{figure}
\centering
\includegraphics[width=\hsize]{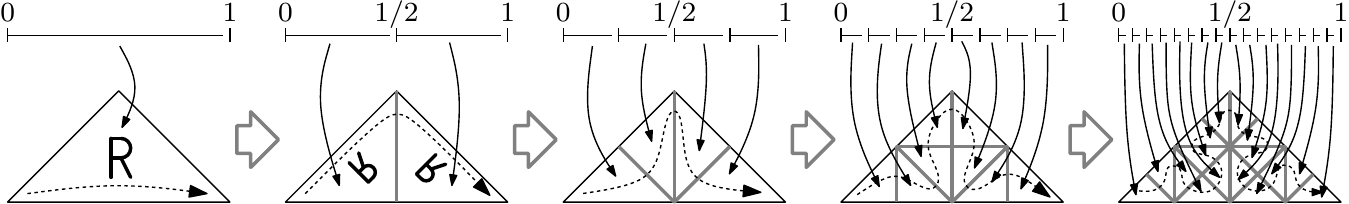}
\caption{Definition and increasingly fine approximations of the Sierpi\'nski traversal order. The dotted curve through the tiles in each figure illustrates the order in which the tiles are traversed.}
\label{fig:sierpinski}
\end{figure}

In general, a traversal order for a gentile $T$ could be defined recursively as follows. Suppose we are given a gentiling of $T$ with tiles $T_1,...,T_r$, with corresponding traversal orders $f_1,...,f_r$. Let $a_i$ be the total size of the tiles before $T_i$, that is, $a_i = \sum_{j=1}^{i-1} |T_j|$. Then we define a traversal order $f$ for $T$ by $f(x) = f_i((x-a_i)/(a_{i+1}-a_i))$ where $i$ is the maximum $i$ such that $a_i \leq x$. In other words, we define $f$ by transforming the domains of $f_1,...,f_r$ so that their domains become consecutive and cover $[0,1)$. Note that the base case of the recursion is missing. A practical solution is to require that $f$ is order-preserving and self-similar, that is, we define each of the traversal orders $f_i$ to be equal to $t_i \circ f$, where $t_i$ is a similarity transformation that transforms $T$ into $T_i$. (In Figure~\ref{fig:sierpinski}, these transformations are depicted by showing, in each tile, the image of a letter \textsf{R} in the initial triangle.) The traversal order $f$ is then the unique solution of the equation $f(x) = t_i(f((x-a_i)/(a_{i+1}-a_i)))$, where $i$ is the maximum $i$ such that $a_i \leq x$. One may also consider reversing the direction of the traversal within certain tiles as compared to the traversal of $T$, resulting in a traversal that is self-similar but not order-preserving. A more general solution is to require that for any tile $T_i$ in a gentiling of $T$, we have some way of choosing a gentiling ${\cal T}$ to subdivide $T_i$ and to choose an order on the tiles of ${\cal T}$. This enables us to approximate $f$ to any desired level of accuracy.

\paragraph{Space-filling curves}
We may extend the domain of a traversal order to $[0,1]$ by defining $f(1) = \lim_{x \uparrow 1} f(x)$. A measure-preserving \emph{space-filling curve} is a traversal order that is continuous. This can be achieved by choosing the order of the tiles in the gentilings that are applied carefully, to ensure that we have $f_i(1) = f_{i+1}(0)$, for $1 \leq i < r$. In the case of order-preserving self-similar curves, this can be rewritten as $t_i(f(1)) = t_{i+1}(f(0))$. The existence of space-filling curves for two-dimensional tiles was first demonstrated by Peano~\cite{peano}. Since then, various space-filling curves have been defined (see, for example, Sagan~\cite{sagan}). The example in Figure~\ref{fig:sierpinski} describes the well-known Sierpi\'nski curve.

At this point we can also define formally what it means for a space-filling curve $f$ with image $T$ to be \emph{based} on a reptiling or gentiling. We say $f$ is based on a \emph{gentiling} if there is an infinite sequence of gentilings ${\cal T}_1,{\cal T}_2,...$ of $T$, such that (i) each gentiling ${\cal T}_{i+1}$, for $i \geq 1$, can be obtained by subdividing the tiles of ${\cal T}_i$ by a gentiling, and (ii) the curve $f$ traverses the tiles of each gentiling ${\cal T}_i$ one by one. We say $f$ is based on a \emph{reptiling} if there is a natural number $r$ such that ${\cal T}_1$ and the gentilings that are applied to the tiles of each ${\cal T}_i$ to obtain ${\cal T}_{i+1}$ are $r$-reptilings. Note that we do not require these $r$-reptilings to be similar to each other, we only require them to have the same number of tiles $r$. 

\paragraph{Face-continuous space-filling curves}
Observe that on any level of recursion that defines the Sierpi\'nski curve, each pair of triangles that are consecutive in the order share an edge. This property is captured by the concept of \emph{face-continuity}: we say a space-filling curve $f$ is \emph{face-continuous} if it is measure-preserving and the interior of the image of any interval under $f$ is a connected set. (The word ``face-continuity'' stems from its use in the three-dimensional setting, where it captures the fact that each pair of consecutive polyhedral tiles shares a two-dimensional face.)

\subsection{Reptilings and space-filling curves based on triangles: our results}

\paragraph{Problem statement}
Several space-filling curves have been described in the literature, and they differ in usefulness, depending on the application. For example, the Lebesgue curve~\cite{lebesgue} (the curve underlying Morton indexing~\cite{morton}), the Hilbert curve~\cite{hilbert}, and the Sierpi\'nski curve have all been used for applications in which the other two curves would not work as well~\cite{quadtree,rtree,tsp}. Thus there is a general interest in exploring the space of possible space-filling curves for curves with potentially useful properties.

One particular use of the Sierpi\'nski curve is in the construction and traversal of two-dimensional triangular meshes as described by Bader and Zenger~\cite{bader}. The Sierpi\'nski curve is particularly well-suited for this application for two reasons. First, the curve is face-continuous, which was exploited by Bader and Zenger to reduce the number of stack operations in their algorithm. Second, the curve is based on a reptiling of triangles that are reasonably well-shaped for use as elements in a computational mesh: the smallest angles of the triangles are not that small; their largest angles are not that large; and their longest edges are not that long (relative to the square root of the area of a triangle). However, the shape of the right triangles of the Sierpi\'nski curve is not optimal in these respects: (certain) acute triangles would be better. This raises the question whether one can construct a face-continuous space-filling curve based on reptilings with acute triangles, which could then replace the Sierpi\'nski curve in the traversal method described by Bader and Zenger~\cite{bader}. This paper investigates that question.

We may address this question in three stages. First, what triangles are $r$-reptiles for what $r$? Second, what tilings are admitted by those triangles? Third, which of these tilings support the definition of a face-continuous space-filling curve?

\begin{figure}
\centering
\includegraphics{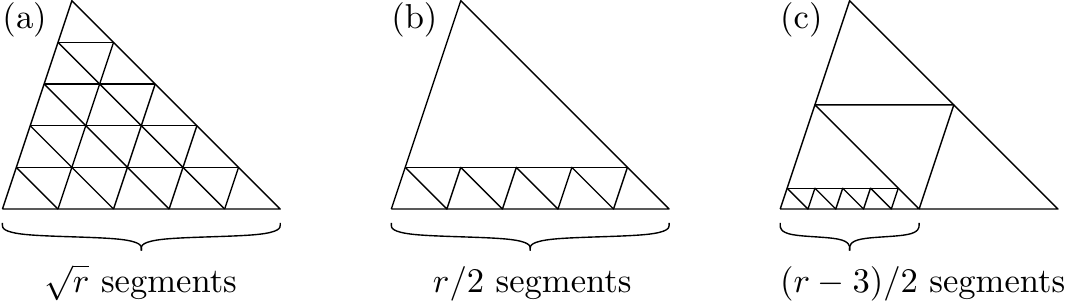}
\caption{(a) A trivial $r$-reptiling for any square number $r = n^2$. (b) A trivial $r$-gentiling for any even $r \geq 4$. (c) A trivial $r$-gentiling for any odd $r \geq 7$.}
\label{fig:trivial}
\end{figure}

\begin{inset}[b]
\caption{Which triangles are $r$-gentiles and $r$-reptiles for what $r$?}
\footnotesize
As shown by Freese et al.~\cite{freese} and illustrated by Figure~\ref{fig:trivial}(b,c), each triangle is an $r$-gentile if $r = 4$ or $r \geq 6$. Figure~(a) below shows a 4-gentiling that is non-trivial for any non-equilateral triangle. Figure~(b) below shows that each right triangle is a 2-gentile, and thus, by repeatedly subdividing a single tile, one can get an $r$-gentiling of a right triangle for any $r \geq 2$. Freese et al.~\cite{freese} also showed that no oblique triangle is a 2-gentile or a 3-gentile, and Kaiser~\cite{kaiser} proved that the only oblique triangle that is a 5-gentile is the isosceles triangle with top angle $2\pi/3$; the tiling is shown in Figure~\ref{fig:facecontinuouscurves}(d).

\addvspace\baselineskip
\bgroup\centering
\includegraphics[width=0.95\hsize]{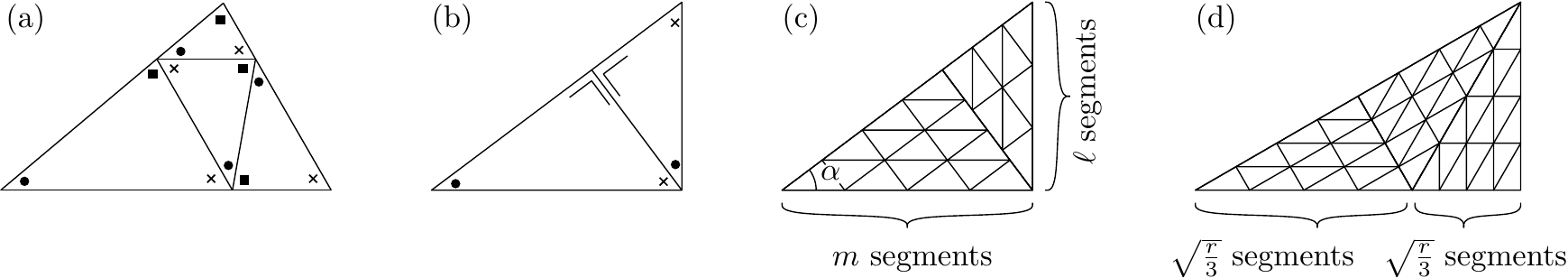}\par
\egroup
\addvspace\baselineskip

Snover et al.~\cite{snover} showed that a triangle $T$ is an $r$-reptile if and only if at least one of the following conditions is satisfied: (i) $r$ is a square number; (ii) $T$ is a right triangle with angle $\alpha$ and there are natural numbers $l$ and $m$ such that $\tan\alpha = l/m$ and $r = l^2 + m^2$; (iii) $T$ is a right triangle with angles $\pi/6$, $\pi/3$ and $\pi/2$ and $r/3$ is square. Examples of corresponding reptilings are illustrated in Figure~\ref{fig:trivial}(a), Figure (c) above, and Figure (d) above, respectively, but alternative tilings with the same numbers of tiles may exist.
\label{ins:gentiles}
\end{inset}

\paragraph{What triangles are reptiles?}
The answer to the first question is known: it is known what triangles are $r$-reptiles for what $r$. To start with, every triangle $T$ admits an $n^2$-reptiling for any natural $n$. Such a tiling is achieved as follows: we divide every side of $T$ into $n$ segments of equal length, and then we draw all lines that are parallel to the sides of $T$ and contain an endpoint of one of the segments, see Figure~\ref{fig:trivial}(a). We will call the resulting subdivision of $T$ the \emph{trivial} $n^2$-reptiling. In fact, we can also use such reptilings to obtain various gentilings. By merging some of the tiles into a larger triangle, we can obtain an $r$-gentiling of $T$ for any $r \geq 6$, as shown by Freese et al.~\cite{freese}, see Figure~\ref{fig:trivial}(b,c). We will call such gentilings \emph{trivial} as well, that is, a gentiling is trivial if there is an underlying set of triangles in a regular grid pattern as described above, such that each tile is a union of triangles from this grid. As far as the value of $r$ goes, these are the only possibilities for acute triangles: acute triangles admit an $r$-reptiling if and only if $r$ is square, and they admit an $r$-gentiling if and only if $r$ is a positive natural number other than 2, 3, or 5 (for further references, see Inset~\ref{ins:gentiles}).

\begin{inset}
\caption{Can you solve this puzzle before reading this article?}
\leavevmode\hbox{}\par\addvspace{0.5\baselineskip}
\bgroup\centering
\includegraphics[width=0.95\hsize]{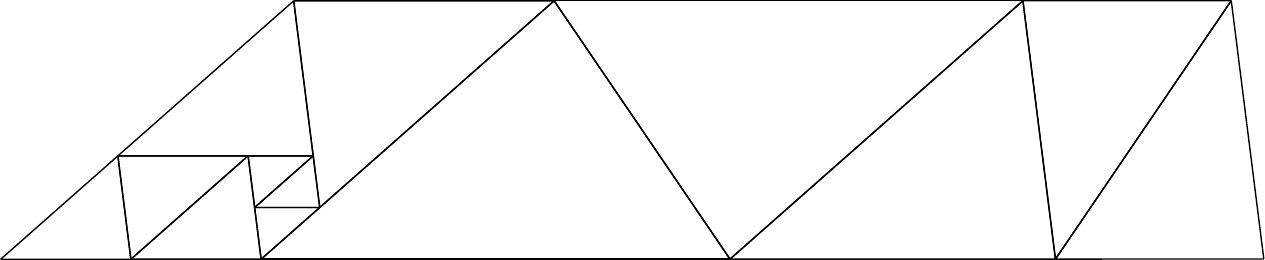}\par
\egroup
\addvspace{.5\baselineskip}
Can you rearrange the triangles in the above figure to form a single triangle? You may need to flip some pieces over.
\label{ins:puzzle}
\end{inset}

\paragraph{What reptilings are possible?}
The answer to the second question is not known: although we know what triangles admit \emph{at least one} $r$-reptiling, we do not know of a full characterization of triangles that admit \emph{non-trivial} $r$-reptilings. Inset~\ref{ins:gentiles} shows some non-trivial reptilings and gentilings but this figure certainly does not illustrate all possibilities. The reader who would like to get some experience with non-trivial tilings, may try to solve the puzzle in Inset~\ref{ins:puzzle} before reading on (at least one solution is a gentiling, which could be refined into a reptiling by cutting up the pieces of the puzzle further). In Section~\ref{sec:irrational} we obtain the following result:

\begin{theorem}\label{main1}
If $T$ is an acute triangle, there exists a non-trivial reptiling of $T$ if and only if $T$ has a pair of sides with lengths $a$ and $b$, such that $b/a$ is a rational number $p/q$ where $p \neq q$.
\end{theorem}

We are interested in the question what non-trivial reptilings are possible for acute triangles particularly in the context of the third question: which tilings support the definition of a face-continuous space-filling curve? Consider a reptiling ${\mathcal T}$ of an acute triangle $T$. If multiple tiles of ${\mathcal T}$ meet at a vertex $v$ of $T$, we say $v$ is a \emph{fan} in ${\mathcal T}$. If there is only a single tile touching a vertex $v$ of $T$ and the interior edge of that tile (the edge opposite to $v$) is a union of edges of multiple adjacent tiles, then we say $v$ is a \emph{cap} in ${\mathcal T}$. As we will see in Section~\ref{sec:sfc}, to construct a face-continuous space-filling curve we need tilings that have caps or fans, but these were absent in all of the reptilings of acute triangles that we have found in the literature. In Section~\ref{sec:cornersplittingtilings} we prove the following result on fans (which, to some extent, also applies to gentilings):

\begin{theorem} \label{main2}
There exists an acute triangle $T$ that admits a reptiling in which one of the vertices of $T$ is adjacent to $k$ tiles if and only if $k = 1$ or $k = 2$.
\end{theorem}

\paragraph{What triangular reptilings support face-continuous space-filling curves?}
Theorem~\ref{main2} confirms the existence of reptilings of acute triangles with fans. However, in Section~\ref{sec:sfc} we find that the conditions of Theorems \ref{main1} and~\ref{main2} are sufficiently restrictive to be able to derive the following negative result:

\begin{theorem} \label{main3}
There is no face-continuous space-filling curve whose construction is based on a reptiling of an acute triangle.
\end{theorem}

\paragraph{Contents of this article}
Further terminology, notation, and a key lemma underlying all of our results are introduced in Section~\ref{sec:notation}.
Theorem~\ref{main1} is proven in Section~\ref{sec:irrational}. In fact, we show that this result partially extends to obtuse triangles as well. Section~\ref{sec:intermezzo} proves that any non-trivial reptiling of an oblique triangle must contain points where a vertex of one tile lies on the interior of an edge of another tile. This is ultimately without consequence for our final results, but may be useful in solving the open problems that remain unsolved by our work. Theorem~\ref{main2} is proven in Section~\ref{sec:cornersplittingtilings}, and Theorem~\ref{main3} is proven in Section~\ref{sec:sfc}. We summarize our results and identify unanswered questions in Section~\ref{sec:furtherresearch}.

\section{Preliminaries}\label{sec:notation}
Let $T$ be a triangle with corners $A,B$ and $C$, and respective angles $\alpha$, $\beta$ and $\gamma$. Consider now a valid $r$-gentiling for $T$ with tiles $T_1, \ldots, T_r$. Call this tiling $\mathcal T$. For more convenient description of $\mathcal T$, we associate with it a graph $G_{\mathcal T}$, which is defined as follows. Each point that is a corner of one or more tiles $T_i$ (for $1\leq i \leq n$) constitutes a vertex of $G_{\mathcal T}$, and each maximal segment of the boundary of $T$ or the boundary between tiles that does not contain a vertex in its interior constitutes an edge of $G_{\mathcal T}$. Note that $G_{\mathcal T}$ is a simple graph, whose vertex set also includes $A,B$ and~$C$.

If we consider a vertex of $G_{\mathcal T}$, we can classify it with respect to the number of tiles that are adjacent to it. We call a tile adjacent to a vertex in two different forms. A vertex $v$ has a \emph{corner-adjacency} to some tile if $v$ is one of the tile's corners. On the other hand, $v$ has an \emph{edge-adjacency} to some tile if $v$ lies on one of the sides of this tile (and is not one of its corners). Thus we can distinguish three types of vertices:
\begin{compactitem}
	\item The \emph{master vertices} are the three vertices $A, B, C$---the corners of $T$.
	\item \emph{Half vertices} are the vertices that either lie on the interiors of the sides of $T$ or that have exactly one edge-adjacency. In the latter case, that is, if the half vertex lies in the interior of $T$, it is also called a \emph{hanging vertex}.
	\item \emph{Full vertices} are vertices which are not master vertices and have no edge-adjacencies.
\end{compactitem}
Note that the union of these three types of vertices is the whole vertex set of $G_{\mathcal T}$. We now refine corner-adjacencies as follows. We say that tile $T_i$ is $\zeta$-adjacent to vertex $v$ if the corner of $T_i$ that coincides with $v$ is of angle $\zeta$ (with $\zeta$ taking values $\alpha, \beta$ and $\gamma$). Additionally, call $\deg_{\zeta}(v)$ the $\zeta$-degree of vertex $v$ and let it denote the number of tiles $\zeta$-adjacent to vertex $v$.

\begin{figure}
\centering
\includegraphics[width=\hsize]{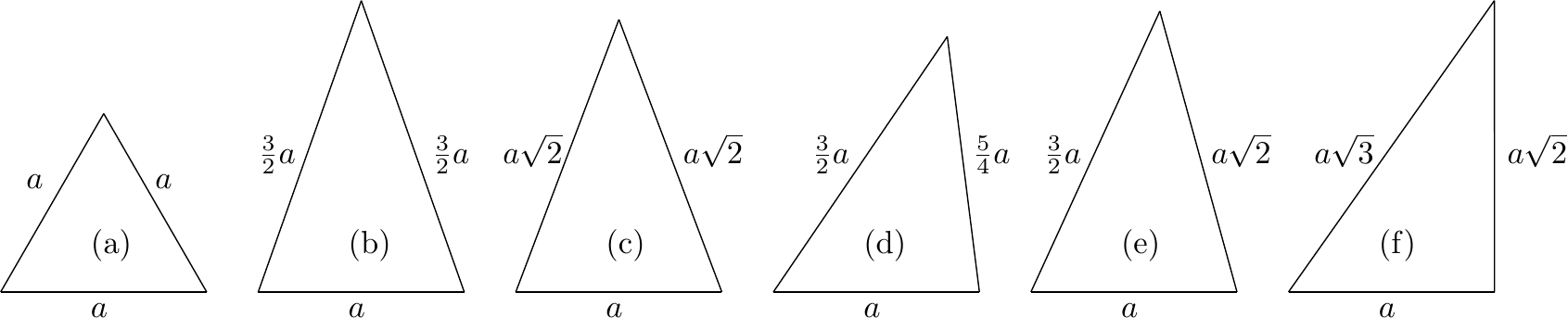}
\caption{Examples of rational triangles (b,d,e) and irrational triangles (a,c,f).}
\label{fig:definitionrationality}
\end{figure}

We call a triangle \emph{rational} if it has at least one pair of sides such that the length of one side divided by the length of the other is a rational number other than one. We call a triangle \emph{irrational} if it is not rational, see Figure~\ref{fig:definitionrationality}

The following consequence of Euler's formula will prove useful in the following sections:

\begin{lemma}\label{lem:euler}
If ${\cal T}$ is an $r$-gentiling with $f$ full vertices and $h$ half vertices, we have $r = 2f + h + 1$.
\end{lemma}
\begin{proof}
Let $e=e(G_{\mathcal T})$ be the number of edges of the graph $G_{\cal T}$. Taking into account the three master vertices and the exterior face of the tiling, Euler's formula gives us $(f + h + 3) + (r + 1) = e + 2$, or equivalently:\begin{equation}\label{eq:euler}
2e = 2r + 2f + 2h + 4
\end{equation}
Each face, that is, each tile as well as the exterior face, has three edges, plus one additional edge for each half vertex that lies in the interior of a side of the face. Since each half vertex lies in the interior of a side of only one face, the number of edges summed over all faces is therefore $3(r + 1) + h$. Note that this counts each edge twice (once from each side), so we have $2e = 3r + h + 3$. Using this to substitute $2e$ in Equation~\eqref{eq:euler}, we get the claimed equality $r = 2f + h + 1$.
\end{proof}

\section{Triangles that admit non-trivial reptilings}
\label{sec:irrational}
In this section we prove Theorem~\ref{main1}: an acute triangle admits a non-trivial reptiling if and only if it is rational. Since the trivial reptiling has neither caps nor fans, this implies that no face-continuous space-filling curves can be constructed based on recursively reptiling an irrational acute triangle. In fact, as we will see below and in Section~\ref{sec:furtherresearch}, the results also extend to certain classes of obtuse triangles.

Throughout this section, we maintain the following notation. We consider an $n^2$-reptiling of a triangle $T$ with vertices $A$, $B$ and $C$ and respective angles $\alpha$, $\beta$ and $\gamma$. The sides of the $T$ opposite of $A, B, C$ are denoted by $BC$, $CA$ and $AB$, respectively, and their lengths are $n\cdot a,n\cdot b,n\cdot c$, respectively. Hence, the small, mutually congruent copies of $T$ that are used to tile $T$ have side lengths $a$, $b$ and $c$. A single such tile is denoted by $T_i$. For $e \in \{BC,CA,AB\}$, an \emph{$e$-parallel} is a line parallel to $e$.

\subsection{A non-trivial tiling of a rational triangle}
\label{sec:rational}

\begin{figure}
  \centering
	\includegraphics{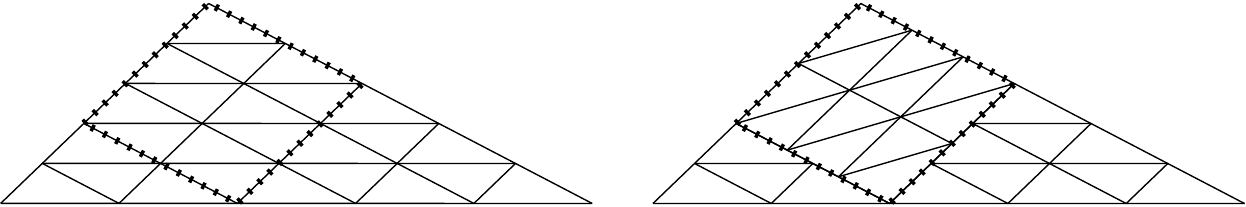}
	\caption{Construction of a non-trivial tiling out of the trivial tiling.} \label{fig:nontriv}
\end{figure}

We start this section by giving a construction for a non-trivial reptiling of any rational triangle, thereby proving that rationality is a sufficient condition for such a reptiling to exist. Let $T$ be a rational triangle, and without loss of generality, assume:
\begin{align*}
p \cdot a = q \cdot b \tag{*}
\end{align*}
for two natural numbers $p > q$. To find a non-trivial tiling, we require $n \geq p+q$. Now consider the trivial $n^2$-tiling of $T$. In this tiling we can find a parallelogram where one side consists of $p$ tile edges of length $a$ and another side consists of $q$ tile edges of length $b$. Due to $n \geq p+q$, this parallelogram must exist. Furthermore, by (*), this parallelogram is a rhombus. We can now ``flip'' the trivial sub-tiling inside the rhombus, that is, mirror it at either of its axes, without interfering with the tiling outside the rhombus. This yields the desired non-trivial tiling. Figure~\ref{fig:nontriv} shows an example for $p=3$ and $q=2$. Thus we get:

\begin{theorem}\label{thm:nontrivialconstruction}
If $T$ is a triangle with a pair of sides with lengths $a$ and $b$, such that $b/a$ is a rational number $p/q$, where $p, q \in \mathbb{N}$ and $p \neq q$, then there is a non-trivial $r$-reptiling of $T$ for any square number $r \geq (p+q)^2$.
\end{theorem}
	
\subsection{All reptilings of irrational acute triangles are trivial}
\label{sec:irrationalistrivial}

We will now prove the following:
\begin{theorem}\label{thm:onlytrivial}
If $T$ is a triangle \emph{without} any pair of sides with lengths $a$ and $b$, such that $b/a$ is a rational number other than 1, and $T$ is (i) acute, or (ii) isosceles and oblique, then all possible reptilings of $T$ are trivial.
\end{theorem}

We start with the following lemma, which gives an important property of the sequences of tiles that may meet along a line segment.

\begin{lemma}
Let $T$ be an irrational scalene triangle. Then there can be only one edge length $e \in \{a,b,c\}$ such that a certain chain of edges of length $e$ has the same length as a certain chain of edges of one or two of the other lengths.
\end{lemma}
\begin{proof}
Assume there are non-zero integers $\lambda_1, \mu_1, \nu_1$ such that
\begin{align} \label{lincomb}
\lambda_1 a + \mu_1 b + \nu_1 c = 0.
\end{align}
Now, suppose there is another triple of integers $(\lambda_2, \mu_2, \nu_2)$  linearly independent of $(\lambda_1, \mu_1, \nu_1)$ such that $\lambda_2 a + \mu_2 b + \nu_2 c = 0$. Eliminating one of the side lengths (say, $c$), we get that $a/b$ is rational, which contradicts the assumption that $T$ is irrational. Hence there can be at most one triple (up to scaling) satisfying Equation~\eqref{lincomb}. This equation can be interpreted as follows: when Equation~\eqref{lincomb} holds, there is exactly one edge length from $\{a,b,c\}$ (namely the one whose coefficient's sign differs from the other two coefficients' signs) such that a multiple of that edge length can be covered with a combination of the other two.
\end{proof}

The above lemma implies that, without loss of generality, any irrational triangle falls into exactly one of two classes as specified by the conditions given below:
\begin{itemize}
\item[(i)] There are \emph{no} non-zero natural numbers $\lambda, \mu, \nu$ such that $\lambda b = \mu c + \nu a$ or $\lambda c = \mu a + \nu b$, and $\beta = \min(\alpha, \beta, \gamma)$ (in other words: if any edge can be written as a rational combination of the others; it is not the shortest edge).
\item[(ii)] There \emph{are} non-zero natural numbers $\lambda, \mu, \nu$ such that $\lambda a = \mu b + \nu c$ and $\alpha = \min(\alpha, \beta, \gamma)$ (in other words: the shortest edge can be written as a rational combination of the other edges).
\end{itemize}
Note that class (i) includes all isosceles irrational triangles.
For class (i), we will show that there can only be the trivial $n^2$-reptiling, provided $T$ is oblique.
For class (ii), we will show that there can only be the trivial $n^2$-reptiling, provided $T$ is acute.
Thus the analysis of these two classes together covers at least all acute irrational triangles and all oblique isosceles irrational triangles.

The following observation will be useful in the analysis of both classes:

\begin{figure}
\centering \includegraphics[height=3cm]{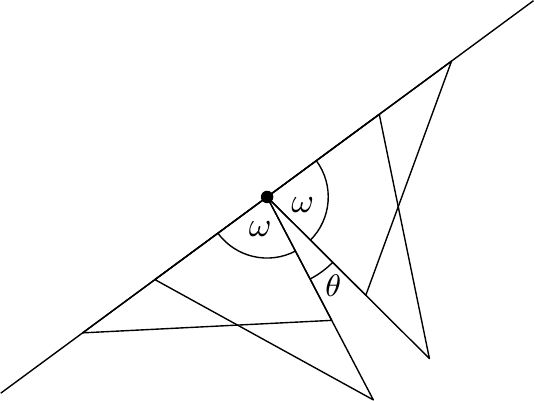}
\caption{A half vertex with two $\omega$-adjacencies, leaving an angle $\theta < \chi$.}
\label{fig:gammameet}
\end{figure}

\begin{observation}\label{obs:gammameet}
If $\chi \leq \psi < \omega$ are the angles of an oblique triangle (which is scalene, or isosceles with top angle $\omega > \pi/3$), and $d_\chi, d_\psi, d_\omega$ are non-negative integers satisfying $d_\chi\cdot\chi + d_\psi\cdot\psi + d_\omega\cdot\omega = \pi$, then we have $d_\omega \leq 1$.
\end{observation}
\begin{proof}
Suppose, for the sake of contradiction, $d_\omega \geq 2$. Then, since $T$ is not a right triangle, $2\omega$ must be strictly smaller than $\pi$, and further angles must be added to make $\pi$. Since $\chi + 2\omega > \chi + \psi + \omega = \pi$, this is not possible. Therefore, $d_\omega \leq 1$.
\end{proof}
Observation~\ref{obs:gammameet} and its application to a half vertex is illustrated in Figure~\ref{fig:gammameet}.

In the analysis of both classes (i) and (ii), we will assume, without loss of generality, that the edge $BC$ of $T$ is horizontal and constitutes the bottom edge of $T$, where $B$ lies to the right of $C$. Let $B_{\varepsilon}(p)$ be the ball with radius $\varepsilon$ centered at $p$. We say a point $p$ is \emph{strictly right-bounded} by an $e$-parallel if there exists some $\varepsilon > 0$ such that in the right half of $B_{\varepsilon}(p)$, the $e$-parallel through $p$ is contained in an edge of ${\cal T}$. We say $p$ is \emph{right-bounded} by an $e$-parallel, if it is strictly right-bounded or there exists some $\varepsilon > 0$ such that in the right half of $B_{\varepsilon}(p)$, the $e$-parallel through $p$ lies outside of $T$. Note that this right half is well-defined after fixing the orientation of our triangle $T$.

Additionally, we say that some tile is \emph{right-adjacent} to a line segment $pq$ if two corners of the tile lie on $pq$, while the third corner lies in the halfplane that is bounded on the left by the line that contains $pq$.

\paragraph{Class (i)} There are no non-zero natural numbers $\lambda, \mu, \nu$ such that $\lambda b = \mu c + \nu a$ or $\lambda c = \mu a + \nu b$, and $\beta = \min(\alpha, \beta, \gamma)$.

\begin{figure}
\centering
\includegraphics{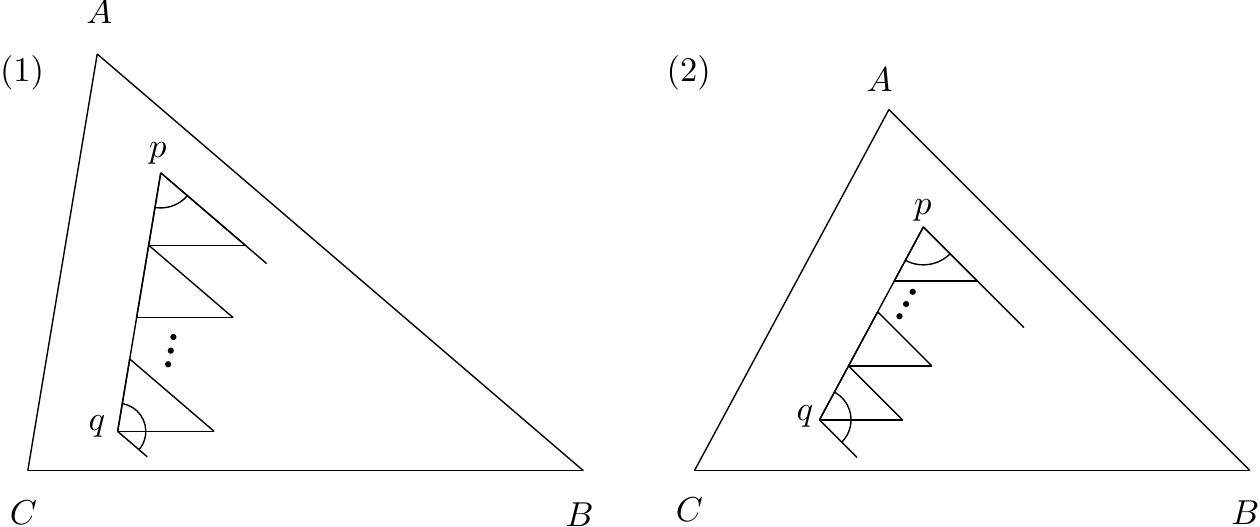}
\caption{Illustration for Lemma~\ref{lem:sawtooth}.}
\label{fig:newsawtooth}
\end{figure}

\begin{lemma}\label{lem:sawtooth}
\textbf{(Sawtooth lemma)} Let ${\cal T}$ be a reptiling of an oblique triangle $T$, where there are no natural numbers $\lambda, \mu, \nu$ such that $\lambda b = \mu c + \nu a$ or $\lambda c = \mu a + \nu b$, and $\beta = \min(\alpha, \beta, \gamma)$. Let $pq$ be a segment of a $CA$-parallel that has length $k \cdot b$ for some $k \in \mathbb{N}$ and is covered completely by edges of ${\cal T}$, where $p$ and $q$ are right-bounded by $AB$-parallels. Then the tiles of ${\cal T}$ that are right-adjacent to $pq$ are translates of scaled (not rotated or reflected) copies of $T$.
\end{lemma}
\begin{proof}
First note that $pq$ cannot be covered by any other combination of edges than $k$ edges of length $b$, of tiles $T_1, ..., T_k$ (from top to bottom). If $T$ is isosceles with $\alpha = \gamma$, the statement immediately follows. In any other case, all of these edges meet, at their endpoints, an $\alpha$ and a $\gamma$ angle of their respective tiles. We can distinguish two cases, as illustrated in Figure~\ref{fig:newsawtooth}: (1) $\alpha < \gamma$, and (2) $\alpha > \gamma$.

In case (1) (if $T$ is isosceles, this means $\alpha = \beta < \gamma$) the angle $\alpha$ between $pq$ and the $AB$-parallel through $p$ forces $T_1$ to be $\alpha$-adjacent to $p$, with the larger $\gamma$-angle at the bottom. By Observation~\ref{obs:gammameet}, two successive tiles $T_i$ and $T_{i+1}$ cannot be $\gamma$-adjacent to the same vertex, so, by induction, each of the tiles $T_i$ has the $\alpha$ angle at the top and the $\gamma$-angle at the bottom. This proves the lemma for this case.

In case (2) (if $T$ is isosceles, this means $\gamma = \beta < \alpha$) we claim that $T_k$ is $\gamma$-adjacent to $q$. If $q$ is not strictly right-bounded, this follows immediately, since $q$ then lies on $BC$ and we have a gap of size $\gamma$. Otherwise, that is, if $q$ is strictly right-bounded, suppose, for the sake of contradiction, that $T_k$ was $\alpha$-adjacent to $q$. The angle $\pi-\alpha$ between $pq$ and the $AB$-parallel through $q$ now yields a gap of $\pi-2\alpha < \pi-\alpha-\gamma = \beta$, which no angle can fill. Hence, $T_k$ is $\gamma$-adjacent to $q$. By Observation~\ref{obs:gammameet}, the $\alpha$-angles of successive tiles $T_i$ and $T_{i+1}$ cannot meet at the same point, so inductively, each of the tiles $T_i$ has the $\gamma$ angle at the bottom and the $\alpha$-angle at the top. This proves the lemma for this case.
\end{proof}

For a point $p$, let its \emph{dominated region} $D(p)$ be the region of $T$ that is both to the left of the $AB$-parallel and to the left of the $CA$-parallel through $p$. For any given set of points $p_1,...,p_m$ in order of decreasing distance from the supporting line of $BC$, let the dominated region $D(p_1,...,p_m)$ be $\bigcup_{i=1}^m D(p_i)$. Let $R$ be the largest dominated region that is tiled exactly as in the trivial tiling, and let $p_1,...,p_m$ be the smallest set of points such that $p_1$ lies on $AB$, $p_m$ lies on $BC$, and $D(p_1,...,p_m) = R$, see Figure~\ref{fig:stairs}. Let $q_i$, for $1 \leq i < m$, be the intersection of the $CA$-parallel through $p_i$ and the $AB$-parallel through $p_{i+1}$. From Lemma~\ref{lem:sawtooth}, with $p = A$ and $q = C$, we know that $R$ contains at least one tile (the tile in the corner at $C$), and therefore $p_m$ does not lie at $C$.

\begin{figure}
\centering
\includegraphics[scale=0.9]{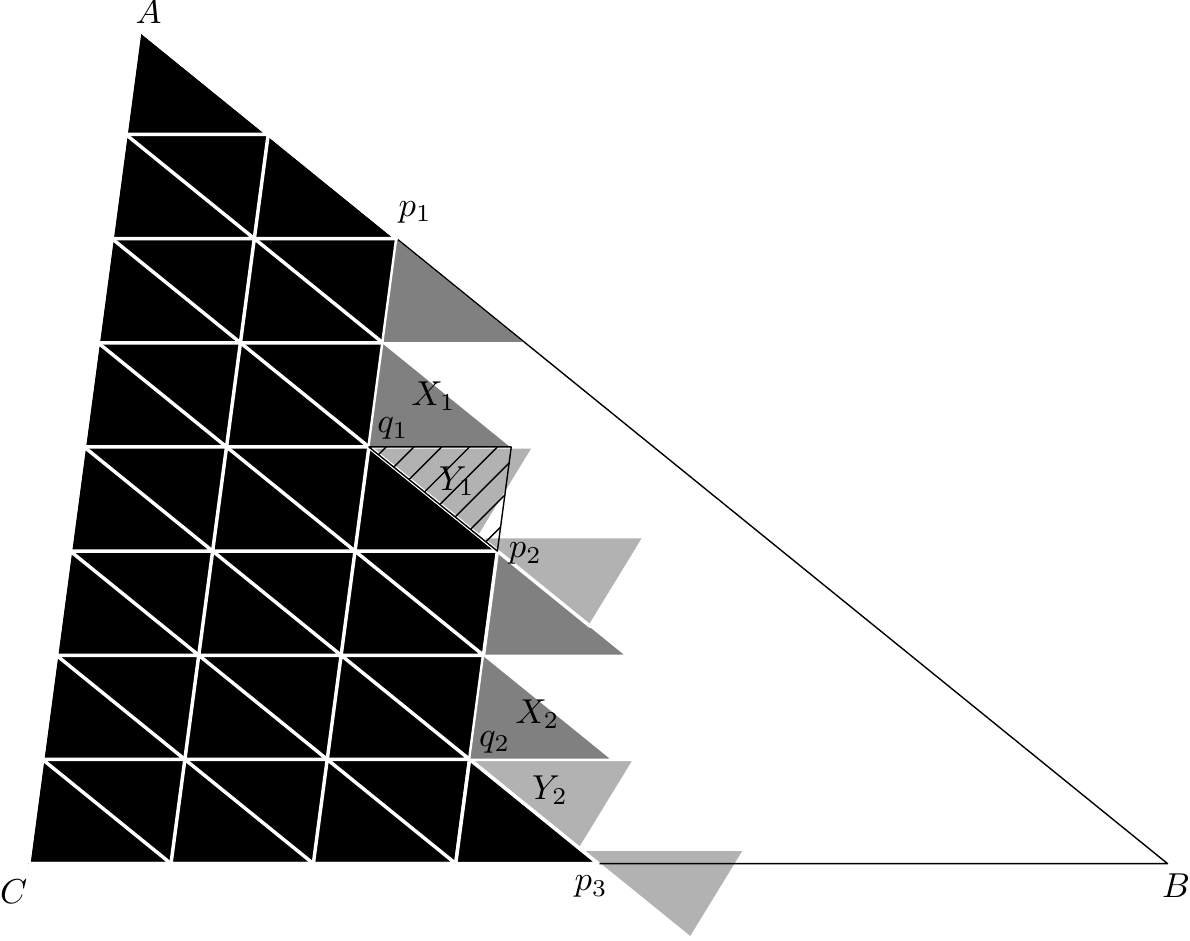}
\caption{Illustration of the analysis of a hypothetical non-trivial tiling under the conditions of class (i). The black region is $R$. The dark gray tiles result from applying Lemma~\ref{lem:sawtooth} (the sawtooth lemma). At $q_1$, the hashed tile cannot be placed because it would violate the maximality of $R$; hence we have to place the gray tile $Y_1$, which results in ``overshooting'' at $p_2$. Applying induction, we find that the configuration of tiles around $q_{m-1}$ induces overshooting at $p_m$, contradicting the very existence of the tiling.}
\label{fig:stairs}
\end{figure}

Now suppose that $p_1 \neq p_m$. We will prove by induction on increasing $i$ that $p_i$ is strictly right-bounded by an $AB$-parallel, which yields a contradiction.
The base case $i = 1$ is established by definition, since $p_1$ lies on $AB$, but $p_1$ is not the master vertex $B$. %
Now consider point $p_i$ with $1<i<m$ and suppose $p_{i-1}$ is strictly right-bounded by an $AB$-parallel. By construction, the length of $p_{i-1}q_{i-1}$ is a multiple of $b$ and immediately below $q_{i-1}$, the tiling ${\cal T}$ corresponds to the trivial tiling. Hence, $q_{i-1}$ is also strictly right-bounded by an $AB$-parallel, and by Lemma~\ref{lem:sawtooth}, a saw-tooth pattern must be placed along $p_{i-1}q_{i-1}$. In particular, $q_{i-1}$ must be $\gamma$-adjacent to a tile $X_{i-1}$ that has exactly the same orientation as $T$ and is also part of the trivial tiling. This only leaves room around $q_{i-1}$ for one more tile $Y_{i-1}$, which must be $\beta$-adjacent to $q_{i-1}$, see Figure~\ref{fig:stairs}. If $Y_{i-1}$ shared an edge of length $a$ with $X_{i-1}$, then $Y_{i-1}$ would also be part of the trivial tiling and $R \cup X_{i-1} \cup Y_{i-1}$ would be a dominated region that is tiled exactly as in the trivial tiling---contradicting the definition of $R$ as the largest such region. Therefore, an edge of length $a$ of $Y_{i-1}$, where $a \neq c$, must lie on the $AB$-parallel through $q_{i-1}$ and $p_i$. (Note that this implies that the triangle cannot be isosceles with $a = c$.) The length of $q_{i-1}p_i$ is a multiple of $c$, which cannot be written as a linear combination of tile edge lengths that includes at least one times $a$. Therefore, the tile edges that are right-adjacent to $q_{i-1}p_i$ cannot cover $q_{i-1}p_i$ exactly and there must be a tile with an edge $e$ that is parallel to $AB$, where the interior of $e$ contains $p_i$. It follows that $p_i$ is also strictly right-bounded by an $AB$-parallel.
It remains to treat the case of $i = m$, but for that, we first need to establish that $p_m \neq q_{m-1}$. Indeed, we just established that $p_{m-1}$ is strictly right-bounded, and from Lemma~\ref{lem:sawtooth}, we get that $q_{m-1}$ is $\gamma$-adjacent to some tile which could be added to $R$ if $p_m = q_{m-1}$. This would contradict the maximality of $R$, so we must have $p_m \neq q_{m-1}$. With that in mind, we can extend our inductive step to the case of $i = m$ and thus get that $p_m$ is strictly right-bounded. But this cannot be true, since $p_m$ lies on the bottom edge $BC$ of $T$. Hence our assumption $p_1 \neq p_m$ must be false and we must have $p_1 = p_m = B$, that is, $R$ includes all of $D(B) = T$: the complete tiling must be trivial.

\paragraph{Class (ii)} There are non-zero natural numbers $\lambda, \mu, \nu$ such that $\lambda a = \mu b + \nu c$ and $\alpha = \min(\alpha, \beta, \gamma)$. The analysis of this class is a little bit more complicated, since we do not have the convenient property anymore that $X_{i-1}$ leaves room around $q_{i-1}$ for only one more tile $Y_{i-1}$. Without loss of generality, we may assume $\beta < \gamma$, and therefore, $\alpha < \beta < \gamma$.

\begin{observation}\label{obs:stillcasei}
The proof for class (i) still goes through verbatim for class (ii) if $\alpha$ does not divide $\beta$, since then, the angle $\beta$ that is left after inserting $X_{i-1}$ at $q_{i-1}$ cannot be filled with smaller angles and can only be filled by a single tile $Y_{i-1}$, as before.
\end{observation}

It remains to analyse the case in which $\alpha$ divides $\beta$. We will use the following observation:
\begin{observation}\label{obs:case3general}
If $\alpha < \beta < \gamma$ are the angles of a scalene oblique triangle, where $\beta/\alpha \in \mathbb{N}$ and $\gamma/\alpha \notin \mathbb{N}$, and $d_\alpha, d_\beta, d_\gamma$ are non-negative integers satisfying $d_\alpha\cdot\alpha + d_\beta\cdot\beta + d_\gamma\cdot\gamma = \pi$, then we have $d_\gamma = 1$ and $d_\beta \leq 1$.
\end{observation}
\begin{proof}
We have $\pi = \alpha + \beta + \gamma = \gamma \neq 0 \pmod\alpha$, while $d_\alpha\cdot\alpha + d_\beta\cdot\beta = 0 \pmod\alpha$, so we must have $d_\gamma \geq 1$, and therefore, by Observation~\ref{obs:gammameet}, $d_\gamma = 1$. It follows that $d_\beta \leq 1$, otherwise we would have $\pi = d_\alpha\cdot\alpha + d_\beta\cdot\beta + d_\gamma \geq 0 + 2\beta + \gamma > \alpha + \beta + \gamma = \pi$.
\end{proof}

In particular, we can make the following observation about acute triangles (in fact, this is the only place in this section where we use the fact that $T$ is not obtuse):
\begin{observation}\label{obs:case3}
If $\alpha < \beta < \gamma$ are the angles of a scalene acute triangle, where $\beta/\alpha \in \mathbb{N}$, and $d_\alpha, d_\beta, d_\gamma$ are non-negative integers satisfying $d_\alpha\cdot\alpha + d_\beta\cdot\beta + d_\gamma\cdot\gamma = \pi$, then we have $d_\gamma = 1$ and $d_\beta \leq 1$.
\end{observation}
\begin{proof}
Let $\beta$ be $k\alpha$. We have $\pi = \alpha + \beta + \gamma > \alpha + 2\beta = (2k+1)\alpha$, but also $\pi = 2(\alpha + \beta + \gamma) - \pi < 2(\alpha + \beta + \pi/2) - \pi = 2(\alpha + \beta) = (2k+2)\alpha$, so $\pi$ is not a multiple of $\alpha$ and thus neither is $\gamma$. Hence we have $d_\gamma = 1$ and $d_\beta \leq 1$ by Observation~\ref{obs:case3general}.
\end{proof}

\begin{figure}
\centering
\includegraphics[scale=0.9]{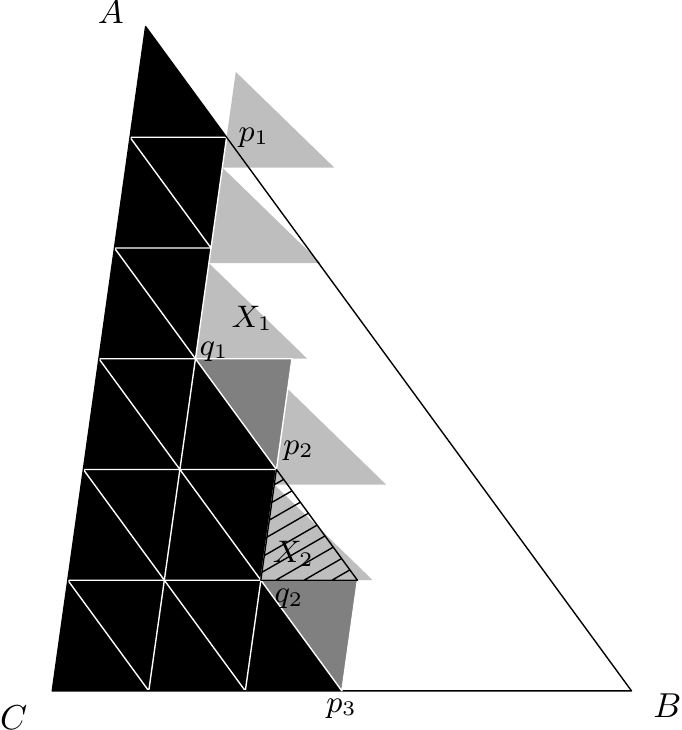}
\caption{Illustration of the analysis of a hypothetical non-trivial tiling under the conditions of class (ii). The black region is $R$. The dark gray tiles result from the fact that the points $p_i$ are strictly right-bounded by a $CA$-parallel, so these tiles have their $\alpha$-angle at the bottom and their sides of length $c$ on $p_i q_{i-1}$. Going bottom-up, we find that the hashed option of placing tile $X_2$ is not possible because it would violate the maximality of $R$; hence we have to place the gray tile, which results in ``overshooting'' at $p_2$. Applying induction and working upwards, we find that $p_1$ is also strictly right-bounded by a $CA$-parallel, contradicting the very existence of the tiling.}
\label{fig:stairs2}
\end{figure}

Now let $R = D(p_1,...,p_m)$ be defined as before, and suppose, for the sake of contradiction, that $m \neq 1$, and therefore, $p_m \neq B$. For an illustration, refer to Figure~\ref{fig:stairs2}. We will prove by induction on \emph{decreasing} $i$ that the points $p_i$ are strictly right-bounded by $CA$-parallels.

We will first establish either the case of $i = m$ or the case of $i = m-1$ as a base case.

If $p_m = q_{m-1}$, we choose $i = m-1$ as the base case of our induction. For the sake of contradiction, suppose $p_m p_{m-1}$ is covered exactly by right-adjacent tiles. Since the length of $p_m p_{m-1}$ is a multiple of $b$, it can only be covered by edges of length $b$, whose respective tiles shall be called $T_1, \ldots, T_h$ with $T_h$ being the tile adjacent to $p_m$. But at $p_m$, there is only a gap of angle $\gamma$ left and since $\alpha$ does not divide $\gamma$, $T_h$ must be $\gamma$-adjacent to $p_m$ and thus could be added to $R$, contradicting the maximality of $R$. Therefore, we must conclude that $p_m p_{m-1}$ is not covered exactly by right-adjacent tiles, and hence $p_{m-1}$ is strictly right-bounded by a $CA$-parallel. This establishes the base case $i = m-1$.

Otherwise, if $p_m \neq q_{m-1}$ (see Figure~\ref{fig:stairs2}), we choose $i = m$ as the base case of the induction. Vertex $p_m$ lies on $BC$. The length of the line segment $q_{m-1} p_{m}$ is a multiple of $c$, which can only be covered by edges of length $c$, at the ends of which their respective tiles $T_1,...,T_h$ (from top to bottom) have angles $\alpha$ and $\beta$. Since $p_m$ has a $\beta$-adjacency from the trivial tiling on the left, by Observation~\ref{obs:case3} it cannot have another $\beta$-adjacency, and therefore it must have an $\alpha$-adjacency from $T_h$. It follows that $p_m$ is bounded by a $CA$-parallel.

Having established the base case, we will now describe the induction step. Suppose $p_{i+1}$ is bounded by a $CA$-parallel. The length of the line segment $q_i p_{i+1}$ is a multiple of $c$, which can only be covered by edges of length $c$, at the ends of which their respective tiles $T_1,...,T_h$ (from top to bottom) have angles $\alpha$ and $\beta$. Between $q_i p_{i+1}$ and the $CA$-parallel through $p_{i+1}$ there is only an angle $\alpha$, and Observation~\ref{obs:case3} says that the $\beta$-angles of two tiles from $T_1,...,T_h$ cannot meet at the same point on $q_i p_{i+1}$. Therefore all tiles $T_1,...,T_h$ must have their $\alpha$-angle at the bottom and their $\beta$-angle on the left. It follows that $q_i$ is bounded by a $BC$-parallel. This leaves a gap of $\gamma$ at $q_i$. Since $\alpha$ does not divide $\gamma$, this can only be filled by a tile $X_i$ with angle $\gamma$ at $q_i$. Now the side of $X_i$ of length $b$ cannot lie on $p_i q_i$, because then $R \cup X_i \cup T_1$ would be a dominated region that is tiled exactly as in the trivial tiling---contradicting the definition of $R$ as the largest such region. Therefore, the edge of $X_i$ that lies on $p_i q_i$ does not have length $b$. The length of $p_i q_i$ is a multiple of $b$, which cannot be written as a linear combination of tile edge lengths that includes at least one times $a$ or $c$. Therefore, the tile edges that cover $p_i q_i$ from the right cannot cover $p_i q_i$ exactly and must ``overshoot'' at $p_i$, so that $p_i$ lies on the interior of a tile edge that is parallel to $CA$. It follows that $p_i$ is also strictly right-bounded by a $CA$-parallel.

Thus, by induction, we find that $p_1$ is strictly right-bounded by a $CA$-parallel. But this cannot be, since $p_1$ lies on $AB$. Hence our assumption $p_1 \neq p_m$ must be false and we must have $p_1 = p_m = B$, that is, $R$ includes all of $D(B) = T$: the complete tiling must be trivial.

This concludes the proof of Theorem~\ref{thm:onlytrivial}, and thus, of Theorem~\ref{main1}: an acute triangle $T$ admits a non-trivial reptiling if and only if $T$ is rational.

\section{Intermezzo: all non-trivial reptilings of oblique triangles have hanging vertices}\label{sec:intermezzo}
From the previous section we learn that we need to focus on triangles of which the length ratio of two sides is a rational number other than one. In the present section we will learn more about the nature of the reptilings we need.
The reader may feel free to skip this section, as the results are ultimately without consequence for our final results---but they can be of interest to those who would like to solve the open problems stated in Section~\ref{sec:furtherresearch} and want to be able to quickly recognize infeasible solutions.

In this section we will prove the following theorem.

\begin{theorem}\label{thm:withouthalfverticesonlygrid}
Any non-trivial reptiling of an oblique triangle must have half vertices in the interior of the triangle.
\end{theorem}

Without loss of generality, assume $\alpha \leq \beta \leq \gamma$. We start the proof of Theorem~\ref{thm:withouthalfverticesonlygrid} with the following lemma.

\begin{lemma}\label{lem:parity}
For any full vertex $v$ in a reptiling without interior half vertices of a scalene triangle, we have $\deg_{\alpha}(v) = \deg_{\beta}(v) = \deg_{\gamma}(v) \pmod 2$.
\end{lemma}

\begin{figure}
\centering
\includegraphics{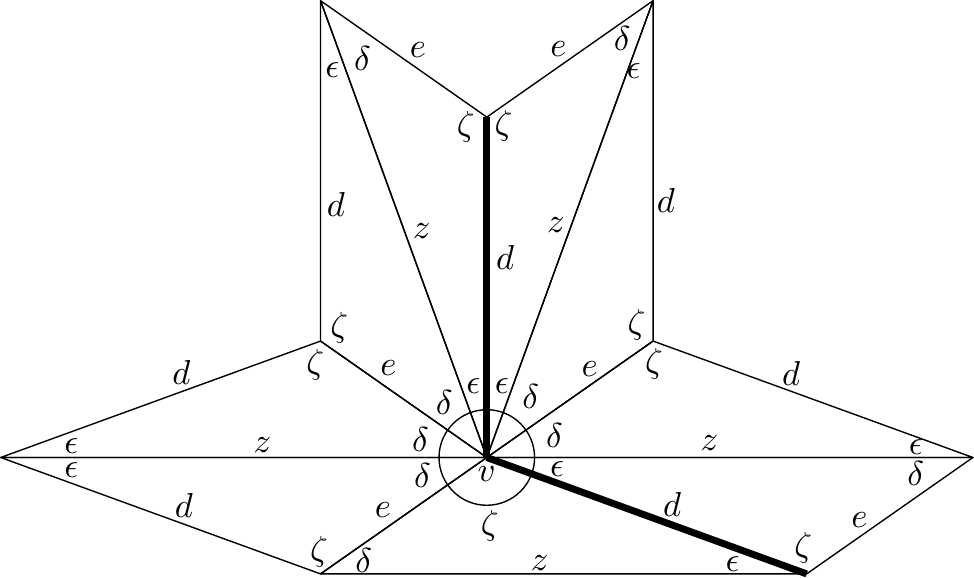}
\caption{Example of a possible sequence of edges and angles incident to a full vertex $v$. Starting from the vertical edge of length $d$, we first see a sequence of the first type: $d, \epsilon, z$, followed by one occurrence of the sequence $\delta, e, \delta, z$, and finally $\epsilon, d$. We then see a sequence of the second type, in which $(\epsilon,e)$ and $(\zeta,z)$ have switched roles: first $d, \zeta, e$, followed by one occurrence of the sequence $\delta, z, \delta, e$, and finally $\delta, z, \epsilon, d$.}
\label{fig:edgeanglecycle}
\end{figure}

\begin{proof}
We consider triangles to have angles $\delta$, $\epsilon$, and $\zeta$, and the lengths of the edges opposite of these angles shall be denoted by $d$, $e$, and $z$, respectively. Now consider the sequence of edges and angles incident and adjacent to a full vertex $v$, in clockwise order around $v$, starting from an arbitrary edge incident to $v$ and ending with that same edge (see Figure~\ref{fig:edgeanglecycle}). We describe such a sequence by the lengths of the edges and the sizes of the angles. Without loss of generality, let the sequence start with $d, \epsilon, z$. It is now easy to verify that the initial part of the sequence up to the next occurrence of $d$ can be completed only in the following ways:\begin{itemize}
\item $d, \epsilon, z$; followed by zero or more occurrences of the sequence $\delta, e, \delta, z$; and finally $\epsilon, d$;
\item $d, \epsilon, z$; followed by zero or more occurrences of the sequence $\delta, e, \delta, z$; and finally $\delta, e, \zeta, d$.
\end{itemize}
The full sequence of edges and angles incident and adjacent to $v$, which returns to the starting edge, is thus composed of sequences of the two types given above, where $(\epsilon, e)$ and $(\zeta, z)$ may (but do not have to) change roles after each occurrence of $d$. Note that for both types of subsequences, the number of angles of each type ($\epsilon$, $\delta$ and $\zeta$) is of the same parity. The lemma follows.
\end{proof}

Recall that from the fact that the triangles are acute and scalene, we have $\pi/3 < \gamma < \pi/2$ and $\gamma - \beta < \alpha$.

\begin{lemma}\label{lem:maxgamma} Let ${\cal T}$ be a reptiling of a scalene, oblique triangle $T$ without interior half vertices.
\begin{compactitem}
\item[(i)] For any full vertex $v$ of ${\cal T}$, we have $\deg_{\gamma}(v) \leq 2$.
\item[(ii)] For any boundary vertex $u$ of ${\cal T}$, we have $\deg_{\gamma}(u) \leq 1$.
\end{compactitem}
\end{lemma}

\begin{proof}
(i) Let $v$ be a full vertex in the reptiling. Because $T$ is not equilateral, we have $\gamma > \pi/3$. Therefore $\deg_{\gamma}(v) \leq 5$, so we only need to analyze the cases $\deg_{\gamma}(v) = 3$, $\deg_{\gamma}(v) = 4$, and $\deg_{\gamma}(v) = 5$. For each of these cases we will show that the angles meeting in $v$ cannot sum up to $2\pi$.\\
$\deg_{\gamma}(v) = 3$: By Lemma~\ref{lem:parity}, we know that $\deg_{\alpha}(v)$ and $\deg_{\beta}(v)$ must each be either 1 or at least 3. If $\deg_{\alpha}(v) = \deg_{\beta}(v) = 1$, the angles around $v$ sum up to $\alpha + \beta + 3\gamma = \pi + 2\gamma \neq 2\pi$ (since $T$ is not a right triangle). However, if $\deg_{\alpha}(v) \geq 3$ or $\deg_{\beta}(v) \geq 3$, the angles around $v$ sum up to at least $3\alpha + \beta + 3\gamma = 3(\alpha + \beta + \gamma) - 2\beta > 2\pi$.\\
$\deg_{\gamma}(v) = 4$: By Lemma~\ref{lem:parity}, we know that $\deg_{\alpha}(v)$ and $\deg_{\beta}(v)$ must each be either 0 or at least 2.
If $\deg_{\alpha}(v) = \deg_{\beta}(v) = 0$, the angles around $v$ sum up to $4\gamma \neq 2\pi$ (since $T$ is not a right triangle). However, if $\deg_{\alpha}(v) \geq 2$ or $\deg_{\beta}(v) \geq 2$, the angles around $v$ sum up to at least $2\alpha + 4\gamma > 2\alpha + 2\beta + 2\gamma = 2\pi$.\\
$\deg_{\gamma}(v) = 5$: By Lemma~\ref{lem:parity}, we know that $\deg_{\alpha}(v)$ and $\deg_{\beta}(v)$ must each be at least 1. Therefore the angles around $v$ sum up to at least $\alpha + \beta + 5\gamma = \pi + 4\gamma > 7\pi/3$, which is too much.\\
This proves part (i) of the lemma.

(ii) This is just restating Observation~\ref{obs:gammameet}.
\end{proof}

\begin{lemma}\label{lem:pigeonholegamma}
Let ${\cal T}$ be a reptiling of a scalene, oblique triangle without interior half vertices.
\begin{compactitem}
\item[(i)] For any full vertex $v$ of ${\cal T}$, we have $\deg_{\gamma}(v) = 2$.
\item[(ii)] For any boundary vertex $u$ of ${\cal T}$ that is not one of the master vertices $A$ and $B$, we have $\deg_{\gamma}(u) = 1$.
\end{compactitem}
\end{lemma}

\begin{proof}
Lemma~\ref{lem:euler} implies that the $r$ angles of size $\gamma$ in the tiles are exactly accounted for if each interior vertex accommodates two of them, each boundary vertex fits one, and one of the master vertices fits one more. By Lemma~\ref{lem:maxgamma}, this is also the maximum possible for each of these vertices. Thus, the pigeon hole principle leaves no room for any vertex to have fewer incident angles of size $\gamma$. Lemma~\ref{lem:pigeonholegamma} follows.
\end{proof}

\begin{lemma}\label{lem:maxbeta}
Let ${\cal T}$ be a reptiling of a scalene, oblique triangle without interior half vertices.
\begin{compactitem}
\item[(i)] For any full vertex $v$ of ${\cal T}$, we have $\deg_{\beta}(v) \leq 2$.
\item[(ii)] For any boundary vertex $u$ of ${\cal T}$, we have $\deg_{\beta}(u) \leq 1$.
\end{compactitem}
\end{lemma}

\begin{proof}
(i) Let $v$ be a full vertex of ${\cal T}$. By Lemma~\ref{lem:pigeonholegamma}, we have $\deg_{\gamma}(v) = 2$. Suppose $\deg_{\beta}(v) > 2$, and thus, by Lemma~\ref{lem:parity}, $\deg_{\beta}(v) \geq 4$. Then the angles around $v$ sum up to at least $4\beta + 2\gamma > 2\alpha + 2\beta + 2\gamma = 2\pi$, which is too much. Hence $\deg_{\beta}(v) \leq 2$, which establishes part (i) of the lemma.

(ii) Let $u$ be a boundary vertex of ${\cal T}$. By Lemma~\ref{lem:pigeonholegamma}, we have $\deg_{\gamma}(u) = 1$. Suppose $\deg_{\beta}(u) > 1$, and thus, $\deg_{\beta}(u) \geq 2$. Then the angles around $u$ sum up to at least $2\beta + \gamma > \alpha + \beta + \gamma = \pi$, which is too much. Hence $\deg_{\beta}(u) \leq 1$, which establishes part (ii) of the lemma.
\end{proof}

\begin{lemma}\label{lem:pigeonholebeta}
Let ${\cal T}$ be a reptiling of a scalene, oblique triangle without interior half vertices.
\begin{compactitem}
\item[(i)] For any full vertex $v$ of ${\cal T}$, we have $\deg_{\beta}(v) = 2$.
\item[(ii)] For any boundary vertex $u$ of ${\cal T}$ that is not one of the master vertices $A,C$, we have $\deg_{\beta}(u) = 1$.
\end{compactitem}
\end{lemma}

\begin{proof}
The proof is completely analogous to the proof of Lemma~\ref{lem:pigeonholegamma}.
\end{proof}

\begin{lemma}\label{lem:allvertices}
Let ${\cal T}$ be a reptiling of a scalene, oblique triangle without interior half vertices.
\begin{compactitem}
\item[(i)] For any full vertex $v$ of ${\cal T}$, we have $\deg_{\gamma}(v) = \deg_{\beta}(v) = \deg_{\alpha}(v)  = 2$.
\item[(ii)] For any boundary vertex $u$ of ${\cal T}$ that is not one of the master vertices, we have $\deg_{\gamma}(u) = \deg_{\beta}(u) = \deg_{\alpha}(u)  = 1$.
\item[(iii)] For the three master vertices of ${\cal T}$, we have, in total, one adjacency of each type $\alpha, \beta$ and $\gamma$.
\end{compactitem}
\end{lemma}

\begin{proof}
The counts of $\gamma$ angles are given by Lemma~\ref{lem:pigeonholegamma}.
The counts of $\beta$ angles are given by Lemma~\ref{lem:pigeonholebeta}.
The counts of $\alpha$ angles follow.
\end{proof}

\begin{lemma}\label{lem:allverticesisosceles}
Let ${\cal T}$ be a (hypothetical) non-trivial reptiling of an isosceles, oblique triangle $T$ without interior half vertices. Let $\lambda$ and $\alpha$ be the base and top angles of $T$.
\begin{compactitem}
\item[(i)] For any full vertex $v$ of ${\cal T}$, we have $\deg_{\lambda}(v) = 4$ and $\deg_{\alpha}(v) = 2$.
\item[(ii)] For any boundary vertex $u$ of ${\cal T}$ that is not one of the master vertices, we have $\deg_{\lambda}(u) = 2$ and $\deg_{\alpha}(u) = 1$.
\item[(iii)] For the three master vertices of ${\cal T}$, we have, in total, two adjacencies of type $\lambda$ and one adjacency of type $\alpha$.
\end{compactitem}
\end{lemma}
\begin{proof}
By Theorem~\ref{thm:onlytrivial}, $T$ must be rational, that is, $\cos\lambda$ must be rational. Furthermore $\lambda \neq \pi/3$, since then $T$ would be equilateral and only admit trivial reptilings. Therefore, by Niven's theorem, $\lambda/\pi$ is irrational.

We will now prove part (i) of the lemma; parts (ii) and (iii) are analogous. Let $v$ be a full vertex of ${\cal T}$. We have $\deg_{\lambda}(v) \cdot \lambda + \deg_{\alpha}(v) \cdot \alpha = 2\pi$, and $4\lambda + 2\alpha = 2\pi$. Suppose, for the sake of contradiction, that $(\deg_{\lambda}(v), \deg_{\alpha}(v)) \neq (4,2)$. Then, solving for $\lambda$, we find that $\lambda/\pi$ is a rational number, which contradicts our conclusion that $\lambda/\pi$ is irrational. Therefore we must have $(\deg_{\lambda}(v), \deg_{\alpha}(v)) = (4,2)$.
\end{proof}

\begin{figure}
\centering
\includegraphics{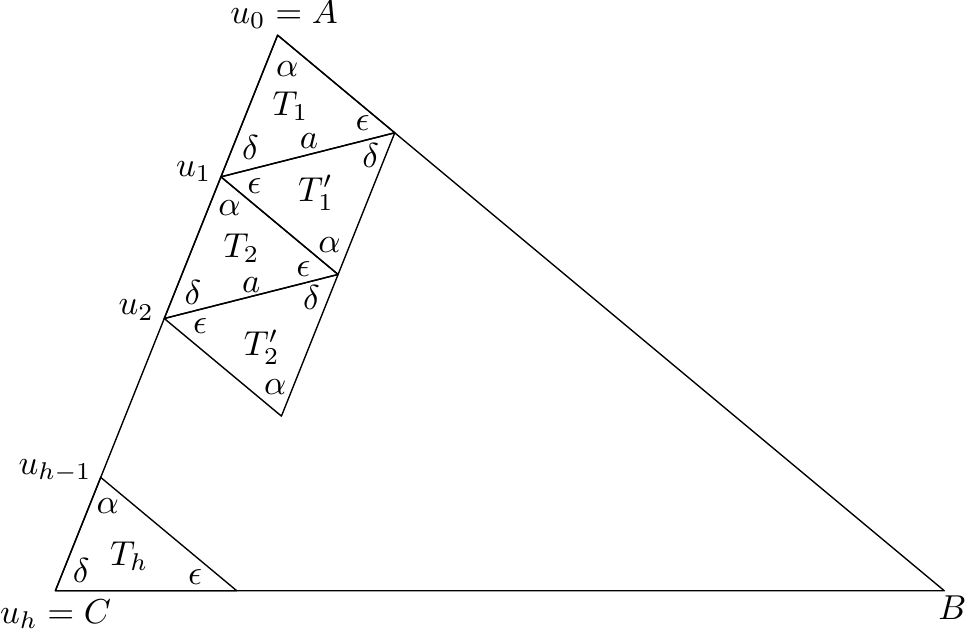}
\caption{Illustration of the construction in the proof of Theorem~\ref{thm:withouthalfverticesonlygrid}.}
\label{fig:conformalistrivial}
\end{figure}

\begin{proof}[Proof of Theorem~\ref{thm:withouthalfverticesonlygrid}]
Let $\mathcal T$ be an $n^2$-reptiling of an oblique triangle $T$ without interior half vertices. We prove that any tile sharing either a single vertex or a whole edge with $CA$ has to be placed exactly as in the trivial tiling. Following this, an inductive argument on the remaining $(n-1)^2$ tiles that must cover the remaining triangle finishes the proof.

Fix $T$ as in Section~\ref{sec:irrationalistrivial}, i.e. $BC$ is fixed horizontally with $B$ to the right of $C$, and $A$ above $BC$. More specifically, if $T$ is isosceles, let $A$ be the top vertex with angle $\alpha$ let $B$ and $C$ the base vertices with angles $\lambda$. Let $T_1, \ldots, T_h$ be the tiles sharing an edge with $CA$, top to bottom (i.e. $T_1$ is adjacent to master vertex $A$). Let $A=u_0, u_1, \ldots, u_h=C$ be the master vertices and half vertices on $CA$ which constitute the corners of the tiles $T_i$ ($1 \leq i \leq h$). Note that due to Lemma~\ref{lem:allvertices}(ii) (if $T$ is scalene) or Lemma~\ref{lem:allverticesisosceles}(ii) (if $T$ is isosceles), every vertex $u_i$ must be adjacent to exactly one more tile for $1 \leq i \leq h-1$. Let $T'_1, \ldots, T'_{h-1}$ be those tiles, where $u_i$ is the unique vertex of $T'_i$ that lies on $CA$, see Figure~\ref{fig:conformalistrivial}.

By Lemma \ref{lem:allvertices}(iii) (if $T$ is scalene) or Lemma~\ref{lem:allverticesisosceles}(iii) (if $T$ is isosceles), $T_1$ must be $\alpha$-adjacent to $A$. Let $\delta$ be the angle of $T_1$ at vertex $u_1$, and let $\epsilon$ be the third angle of $T_1$. We will now show by induction on $i$ that, for $1 < i \leq h$, each $T'_{i-1}$ is a translate of $T_1$ rotated 180 degrees, and each $T_i$ is a translate of $T_1$. We take the second part of this claim for $i = 1$ as the trivial base case. Now, given that $T_{i-1}$ is a translate of $T_1$, with its $\alpha$-angle at $u_{i-2}$ and a $\delta$-angle at $u_{i-1}$, we observe the following. In order to avoid a half vertex in the interior of $T$ along the side of $T_{i-1}$ of length $a$, the tile $T'_{i-1}$ has to share its side of length $a$ with $T_{i-1}$. Thus $T'_{i-1}$ has to be $\epsilon$-adjacent to $u_1$ by Lemma~\ref{lem:allvertices}(ii) (if $T$ is scalene) or simply by the fact that $\delta = \epsilon = \beta = \gamma$ (if $T$ is isosceles). Now $T_i$ must  share the edge opposite of angle $\delta$ with the corresponding edge of $T'_{i-1}$ in order to avoid creating a half vertex, and thus, $T_i$ is $\delta$-adjacent to $u_i$. In particular, $T_h$ is $\delta$-adjacent to $u_h = C$, and therefore, by Lemma \ref{lem:allvertices}(iii), $\delta = \gamma$ and $\epsilon = \beta$. Thus all tiles touching $CA$ are placed exactly as in the trivial tiling. The theorem follows.
\end{proof}

\section{Corner-splitting reptilings}
\label{sec:cornersplittingtilings}

A \emph{$k$-splitting} gentiling ${\cal T}$ of a triangle $T$ is a gentiling of $T$ in which $k$ tiles meet in one of the vertices of $T$, where $k \geq 2$. We say a gentiling is corner-splitting if it is $k$-splitting for some $k \geq 2$. In this section we prove Theorem~\ref{main2}: only for $k = 2$ there exist acute triangles that admit $k$-splitting reptilings. In fact, we prove a little more. In Section~\ref{sec:splittingconstruction} we describe a construction of $k$-splitting gentilings and reptilings that proves the following:

\begin{theorem}\label{thm:gent2split}
\begin{compactitem}
\item[(i)] Any triangle with angles $0 < \alpha < \pi/4$, $2\alpha$ and $\pi-3\alpha$ such that $\cos^2\alpha$ is rational, admits a 2-splitting gentiling.
\item[(ii)] Any triangle with angles $0 < \alpha < \pi/4$, $2\alpha$ and $\pi-3\alpha$ such that $\cos\alpha$ is rational, admits a 2-splitting reptiling.
\end{compactitem}
\end{theorem}

In Section~\ref{sec:nosplitting} we prove that no acute triangle admits a $k$-splitting gentiling with $k \geq 3$ (except, possibly, for two specific triangles that \emph{might} admit a 3-gentiling), and no oblique isosceles triangle admits any corner-splitting reptiling.


\subsection{A construction of 2-splitting gentilings}
\label{sec:splittingconstruction}
We now prove Theorem~\ref{thm:gent2split} by giving an example.

Consider a triangle $T$ with angles $\alpha$, $2\alpha$, and $\pi-3\alpha$. The side lengths have ratios $\sin\alpha : \sin 2\alpha : \sin 3\alpha$, or equivalently, $1 : 2\cos\alpha : (4\cos^2\alpha -1)$. We can also write this as $p : q : r$, where $p < r < 3p$ and $q = \sqrt{p^2 + pr}$. If and only if $\cos^2\alpha$ is rational, we can choose $p$ and $r$ to be natural numbers, and if and only if $\cos\alpha$ is rational as well, we can make sure that $p$, $q$ and $r$ are natural numbers.

\begin{figure}
\centering
\includegraphics{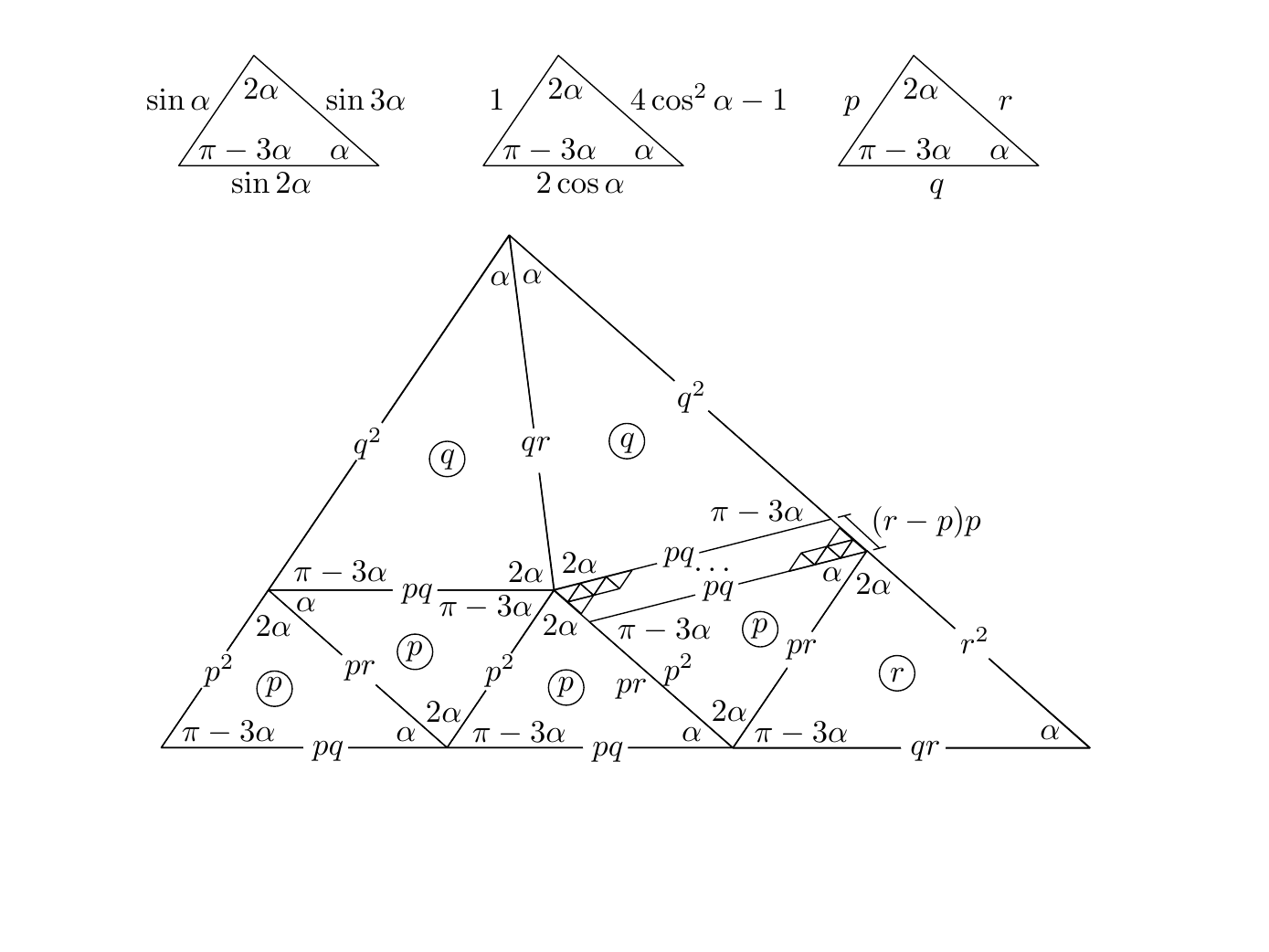}
\caption{Construction of a 2-splitting gentiling for a triangle with angles $\alpha$, $2\alpha$, and $\pi- 3 \alpha$, where $\cos^2\alpha$ is rational. The encircled numbers denote the scale factors of the tiles relative to the smallest triangles in the tiling.}
\label{fig:cornersplittingtiling}

\end{figure}

Figure~\ref{fig:cornersplittingtiling} shows how to assemble seven scaled copies of $T$ (with scale factors $p$, $q$ and $r$) and a parallelogram with $2p(r-p)$ copies of $T$ into a triangle with side lengths $p^2 + q^2 = (2p+r)p$, $(2p+r)q$, and $q^2 + pr - p^2 + r^2 = (2p+r)r$, that is, a copy of $T$ scaled with factor $2p + r$. Note that this construction can be realized if and only if $p$ and $r$ are natural numbers. If, additionally, $q$ is a natural number as well, then the seven scaled triangles can all be tiled with copies of $T$ of the same size that is used to tile the parallelogram, and we obtain a 2-splitting reptiling with $(2p+r)^2$ tiles. The simplest example of such a tiling is obtained with $(p,q,r) = (4,6,5)$, which results in 169 tiles.
(Next best is $(p,q,r) = (9,15,16)$, resulting in $1156$ tiles). This concludes the proof of Theorem~\ref{thm:gent2split}.

Note that the construction is not limited to acute triangles: $p$, $q$ and $r$ can also be chosen such that a reptiling of an obtuse triangle results. For example, with $p = 100$, $q = 190$, and $r = 261$, one gets an obtuse triangle with angles approximately 18.2, 36.4 and 125.4 degrees and a 2-splitting 212\,521-reptiling.


\subsection{Acute triangles do not admit 3-splitting reptilings}
\label{sec:nosplitting}

\begin{observation}\label{obs:alphanodivpi}\label{alphadivphi}
Let $T$ be a scalene, acute triangle with angles $\alpha$, $\delta = k \alpha$, and $\phi = \pi - \alpha - \delta$, for natural $k \geq 2$. Then there exists no natural number $t$ such that $t \alpha = \pi$. In particular, $\phi$ is not a multiple of $\alpha$.
\end{observation}
\begin{proof}
Since $T$ is acute and scalene, we have $0 < |\phi - \delta| < \alpha$, so only $\delta$ is divisible by $\alpha$ and $\phi$ is not. It follows that $\pi = \alpha + \delta + \phi$ is not divisible by $\alpha$.
\end{proof}

\begin{observation}\label{obs:splitone}
If ${\cal T}$ is a $k$-splitting gentiling of a scalene, acute triangle $T$ with angles $\alpha < \beta < \gamma$, then there is a single master vertex $X \in \{B,C\}$ such that the tiles of ${\cal T}$ incident to $X$ are $k$ tiles which are all $\alpha$-adjacent to $X$.
\end{observation}
\begin{proof}
Because $T$ is acute, we have $\beta > \gamma - \alpha$, and therefore we cannot split a large ($\beta$ or $\gamma$) angle into a combination of angles that includes another large angle. Hence the master vertex with angle $\beta$ or $\gamma$ must be split into $k$ small angles.
\end{proof}

\begin{lemma}\label{lem:alphadiv2pi}
If $T$ is a scalene, acute triangle that admits a $k$-splitting gentiling ${\cal T}$ with $k \geq 3$, then there is an odd number $m$ such that $\alpha = 2\pi/m$, where $\alpha$ is the smallest angle of $T$.
\end{lemma}
\begin{proof}
By Observation~\ref{obs:splitone}, there is a $\delta \in \{\beta, \gamma\}$ such that $\delta = k\alpha$. Let $\phi = \pi - \alpha - \delta$ be the remaining angle of $T$.

By Lemma~\ref{lem:euler}, the $\delta$ angles of the tiles are exactly accounted for if each full vertex of $G_{\cal T}$ fits two, each half vertex fits one, and there is one at a master vertex. Thus, if we split the master vertex with angle $\delta$ into $k$ angles of size $\alpha$, we need a half vertex with more than one $\delta$ angle and/or a full vertex with more than two $\delta$ angles to accommodate the $\delta$ angles of all tiles.

A half vertex with more than one $\delta$ angle cannot exist, since a half vertex with only $\delta$ and $\alpha$ angles cannot exist by Observation~\ref{obs:alphanodivpi}, and a half vertex with at least two $\delta$ angles and a $\phi$ angle cannot exist due to $2\delta + \phi > \alpha + \delta + \phi = \pi$.

Now suppose $v$ is a full vertex with at least three $\delta$ angles. We cannot have $\deg_\phi(v) \geq 2$, since $3\delta + 2\phi = 2\delta + 2\phi + k\alpha > 2\pi$. We cannot have $\deg_\phi(v) = 1$ either, since this would imply that {$\deg_\delta(v)\delta + \phi = 2\pi \pmod \alpha$} and therefore $\pi = 2\pi - (\alpha + \delta + \phi) = (\deg_\delta(v)-1)\delta - \alpha = 0 \pmod \alpha$, contradicting Observation~\ref{obs:alphanodivpi}. Therefore we must have $\deg_\phi(v) = 0$, that is, $v$ must have only $\delta$ and $\alpha$ angles, and hence $2\pi$ must be divisible by $\alpha$. Because $\pi$ is not divisible by $\alpha$, the divisor $m$ in $\alpha = 2\pi/m$ must be odd.
\end{proof}

\begin{theorem}\label{thm:splittingangles}
No acute triangle $T$ admits a $k$-splitting gentiling for $k \geq 4$. If $T$ is an acute triangle and admits a 3-splitting gentiling, then the angles of $T$ are either $(\frac2{13}\pi, \frac5{13}\pi, \frac6{13}\pi)$ or $(\frac2{15}\pi, \frac6{15}\pi, \frac7{15}\pi)$.
\end{theorem}
\begin{proof}
Assume $T$ is an acute triangle that admits a $k$-splitting gentiling ${\cal T}$ with $k \geq 3$. Let $\alpha \leq \beta \leq \gamma$ be the angles of $T$. Note that $\alpha < \beta$, otherwise we would have $(\alpha, \beta, \gamma) = (\alpha, \alpha, k\alpha)$ and $\gamma = k\pi/(k+2) > \pi/2$ and $T$ would not be acute.

By Lemma~\ref{lem:euler} we have $2r = 4f + 2h + 2$. This means that the large ($\beta$ and $\gamma$) angles of the tiles in a tiling ${\cal T}$ of $T$ are exactly accounted for if each full vertex of $G_{\cal T}$ fits four, each half vertex fits two, and there are two at the master vertices. Thus, if we replace a large angle at a master vertex by $k$ small ($\alpha$) angles, we need at least one full vertex with more than four large angles or at least one half vertex with more than two large angles to accommodate the large angles of all tiles.

Observe $\gamma \geq \beta = \pi - \gamma - \alpha > \pi/2 - \alpha > (1-1/k) \pi/2$, due to the fact that $T$ is acute. Hence, for $k \geq 3$, we have $\gamma \geq \beta > \frac13\pi$, and no half vertex can have more than two large angles. It follows that there must be a full vertex $v$ with at least five large angles, and therefore, $5\beta \leq 2\pi$.

If $T$ is isosceles, that is, $\gamma = \beta = k\alpha$, then $5\beta \leq 2\pi$ can be rewritten as $5k\pi/(2k+1) \leq 2\pi$, which solves to $k \leq 2$.

If $T$ is not isosceles, we get from Lemma~\ref{lem:alphadiv2pi} that the angles of $T$ are $\alpha = 2\pi/m$, $\delta = 2k\pi/m$, and $\phi = \pi - \alpha - \delta$, for some odd integer $m$. Since $T$ is acute and scalene, we have $0 < |\phi - \delta| < \alpha$. With $\phi - \delta = \pi - \alpha - 2\delta = (m - 2 - 4k)\frac\pi m$ and $\alpha = 2\frac\pi m$ we now get $0 < |m - 2 - 4k| < 2$ for some odd integer $m$, hence $|m - 2 - 4k| = 1$ and $|\phi - \delta| = \pi/m = \alpha/2$. Thus, we either have:
\begin{itemize}
\item $m = 2 + 2k + (2k - 1) = 4k + 1$, where $\beta = \phi = (2k-1)\pi/(4k+1)$ or
\item $m = 2 + 2k + (2k + 1) = 4k + 3$, where $\beta = \delta = 2k\pi/(4k+3)$.
\end{itemize}
In both cases, solving $5\beta \leq 2\pi$ yields $k \leq 3$, and we find that $m$ is either 13 or 15, resulting in the two possible sets of angles as stated in the lemma.
\end{proof}

Indeed, 3-splitting gentilings with the aforementioned angles might be realizable: at least they cannot be ruled out based on pigeon-hole arguments on the angles alone. With angles $(\frac2{13}\pi, \frac5{13}\pi, \frac6{13}\pi)$, one can account for the $\alpha$, $\beta$ and $\gamma$ angles of all tiles by including a full vertex with $(\deg_\alpha, \deg_\beta, \deg_\gamma) = (1,0,4)$ and a full vertex with $(\deg_\alpha, \deg_\beta, \deg_\gamma) = (0,4,1)$. With angles $(\frac2{15}\pi, \frac6{15}\pi, \frac7{15}\pi)$, one can account for the $\alpha$, $\beta$ and $\gamma$ angles of all tiles by including a full vertex with $(\deg_\alpha, \deg_\beta, \deg_\gamma) = (1,0,4)$ and a full vertex with $(\deg_\alpha, \deg_\beta, \deg_\gamma) = (0,5,0)$.

\begin{theorem}\label{thm:no3splittingreptiling}
No acute triangle $T$ admits a $k$-splitting reptiling ${\cal T}$ for $k \geq 3$.
\end{theorem}
\begin{proof}
In this proof, we will use the trigonometric identities $\sin(3\omega) = 3\sin(\omega) - 4\sin^3(\omega)$; $\sin^2(\omega) = \frac12 - \frac12\cos(2\omega)$; and $\sin(3\omega)/\sin(\omega) = 3-4\sin^2(\omega) = 2\cos(2\omega) + 1$.

By Theorem~\ref{thm:splittingangles}, if $k \geq 3$, the angles of $T$ would have to be either $(\frac2{13}\pi, \frac5{13}\pi, \frac6{13}\pi)$ or $(\frac2{15}\pi, \frac6{15}\pi, \frac7{15}\pi)$.
In the first case, the ratios of the side lengths are:\[\def\arraystretch{1.5}\begin{array}{lll}
a/b & = \sin(\frac2{13}\pi)/\sin(\frac5{13}\pi) = -\sin(\frac{15}{13}\pi)/\sin(\frac5{13}\pi) = -2\cos(\frac{10}{13}\pi) - 1 & = 2\cos(\frac3{13}\pi) - 1; \\
b/c & = \sin(\frac5{13}\pi)/\sin(\frac6{13}\pi) = -\sin(\frac{18}{13}\pi)/\sin(\frac6{13}\pi) = -2\cos(\frac{12}{13}\pi) - 1 & = 2\cos(\frac1{13}\pi) - 1; \\
c/a & = \sin(\frac6{13}\pi)/\sin(\frac2{13}\pi) & = \rlap{$\displaystyle{2\cos(\frac4{13}\pi) + 1. }$}\hphantom{\hbox{$\displaystyle{\frac12\sqrt{6/(1 - \cos(\frac4{15}\pi)) - 3} + \frac12. }$}}\\
\end{array}\]
In the second case, the ratios of the side lengths are:\[\def\arraystretch{1.5}\begin{array}{lll}
b/a & = \sin(\frac6{15}\pi)/\sin(\frac2{15}\pi) & = 2\cos(\frac4{15}\pi) + 1; \\
b/c & = \sin(\frac6{15}\pi)/\sin(\frac7{15}\pi) = -\sin(\frac{24}{15}\pi)/\sin(\frac8{15}\pi) =     -2\cos(\frac{16}{15}\pi) - 1 & = 2\cos(\frac1{15}\pi) - 1; \\
c/a & = \sin(\frac7{15}\pi)/\sin(\frac2{15}\pi) = \\ & \hphantom{=} \left(\sin(\frac13\pi)\cos(\frac2{15}\pi) + \cos(\frac13\pi)\sin(\frac2{15}\pi)\right)/\sin(\frac2{15}\pi) = \\ & \hphantom{=} \frac12\sqrt{3\cos^2(\frac2{15}\pi)/\sin^2(\frac2{15}\pi)} + \frac12 =
    \frac12\sqrt{3/\sin^2(\frac2{15}\pi) - 3} + \frac12 & = 
    \frac12\sqrt{6/(1 - \cos(\frac4{15}\pi)) - 3} + \frac12. \\
\end{array}\]
All of these side length ratios are irrational because the cosines on the righthand sides of the equations are irrational by Niven's theorem. Therefore, by Theorem~\ref{main1}, $\mathcal T$ can only be the trivial tiling, and a 3-splitting reptiling cannot exist.
\end{proof}

\begin{theorem}\label{thm:noisoscelesksplittingreptiling}
No oblique isosceles triangle $T$ admits a $k$-splitting reptiling ${\cal T}$ for $k \geq 2$.
\end{theorem}
\begin{proof}
Suppose, for the sake of contradiction, that $T$ is an oblique isosceles triangle admitting a $k$-splitting reptiling ${\cal T}$ for $k \geq 2$. Let $\lambda$ and $\tau$ be the base and top angle of $T$, respectively.
If a $k$-splitting reptiling exists, we must have either $\tau = k \lambda$ or $\lambda = k \tau$.
In the first case, we get $\lambda = \pi/(k+2)$, and in particular $0 < \lambda \leq \pi/4 < \pi/3$.
In the second case, we get $\lambda = k\pi/(2k+1)$, and in particular $\pi/3 < 2\pi/5 \leq \lambda < \pi/2$.
In both cases, by Niven's theorem, we have $\cos(\lambda) \notin \mathbb{Q}$, and therefore, by Theorem~\ref{thm:onlytrivial}, $T$ admits only trivial reptilings and no $k$-splitting reptilings.
\end{proof}

\section{No face-continuous space-filling curves for acute triangles}\label{sec:sfc}
In this section we prove Theorem~\ref{main3}: no face-continuous space-filling curve can be constructed on the basis of reptilings of an acute triangle.

Consider a reptiling ${\mathcal T}$ of an acute triangle $T$. If multiple tiles of ${\mathcal T}$ meet at a master vertex $v$ of $T$, we say $v$ is a \emph{fan} in ${\mathcal T}$. If there is only a single tile touching a master vertex $v$ and the interior edge of that tile (the edge opposite to $v$) is a union of edges of multiple adjacent tiles, then we say $v$ is a \emph{cap} in ${\mathcal T}$. We define the \emph{dual} of ${\mathcal T}$ as the graph that has a node for each tile of ${\mathcal T}$, and contains an edge between two nodes if the corresponding tiles of ${\mathcal T}$ touch each other in more than a single point.

\begin{lemma}\label{lem:hamiltonianpath}
If there is a face-continuous space-filling curve based on a reptiling of $T$, then $T$ must admit a reptiling ${\mathcal T}$ of which the dual contains a Hamiltonian path.
\end{lemma}
\begin{proof}
Consider any tiling ${\mathcal T}$ of $T$ that results from tiling $T$ with reptilings recursively, to a sufficient recursion depth such that no tile touches more than one master vertex. A space-filling curve based on this tiling must visit the tiles one by one. If there is any pair of consecutive tiles along the curve that do not touch each other in more than a single point, then there would be a section of the space-filling curve (namely the section traversing exactly this pair of tiles) with a disconnected interior. Hence the order in which the space-filling curve visits the tiles must be such that the corresponding sequence of vertices in the dual describes a Hamiltonian path.
\end{proof}

\begin{lemma}\label{lem:needfansandcaps}
If there is a face-continuous space-filling curve based on a reptiling of $T$, then $T$ must admit a tiling ${\mathcal T}$ in which at least one master vertex is a cap or a fan.
\end{lemma}
\begin{proof}
Consider any tiling ${\mathcal T}$ of $T$ that results from tiling $T$ with reptilings recursively, to a sufficient recursion depth such that no tile touches more than one master vertex. By Lemma~\ref{lem:hamiltonianpath}, the dual of ${\mathcal T}$ must contain a Hamiltonian path. To be able to construct such a path, we have to be able to get into and out of at least one of the master vertices. More precisely, let $v$ be the vertex of $T$ that is visited second by the space-filling curve. The visit to $v$ must be part of a traversal of a tile $S$ whose corresponding vertex $S'$ in the dual is neither first nor last on the Hamiltonian path. Hence $S'$ must have degree at least two, which is the case exactly if $S$ is one of the tiles that makes $v$ a fan, or if it is the single tile that makes $v$ a cap.
\end{proof}

\begin{lemma}\label{lem:uniquepath}
In a scalene acute triangle, at most one corner is a fan or a cap of any reptiling.
\end{lemma}
\begin{proof}
Since in any triangle, the length difference between the longest two edges is strictly less than the length of the shortest edge, at most one edge of a tile can be exactly as long as two or more edges of other tiles. Therefore, at most one master vertex can be a cap. Similarly, since the largest two angles of any acute triangle differ by strictly less than the smallest angle, $\alpha$, at most one master vertex can be a fan. It remains to show that it cannot be that in one reptiling, one master vertex is a fan, \emph{and} in the same or another reptiling, another master vertex is a cap.

By Theorem~\ref{thm:no3splittingreptiling}, if a fan exists, its angle must be $2\alpha$, where $\pi/6 < \alpha < \pi/4$ and the third master vertex has angle $\pi - 3\alpha$. Now it follows immediately from $\pi/6 < \alpha < \pi/4$, and thus, $1 > \sin(\alpha), \sin(2\alpha), \sin(\pi-3\alpha) > 1/2$, that each edge of a tile is more than half as long as any other edge, and therefore no master vertex can be a cap.
\end{proof}

Let a \emph{two-level reptiling} of a triangle $T$ be an $r^2$-reptiling ${\mathcal T}_2$ that is obtained by first $r$-reptiling a \emph{master} triangle $T$ with \emph{intermediate} tiles, and then $r$-reptiling each intermediate tile with an equal number of \emph{atomic} tiles. The reptilings within the intermediate tiles may differ from one intermediate tile to another, but we require that each of them has the same number of tiles, so all atomic tiles have the same size. Let a \emph{conforming Hamiltonian path} of the dual of ${\mathcal T}_2$ be an ordering of the atomic tiles such that each atomic tile (except the first) touches the previous tile in more than a single point, and the atomic tiles within any intermediate tile are consecutive in the order.

\begin{lemma}\label{lem:noconformingpath}
No acute triangle admits an arbitrarily fine two-level reptiling whose dual admits a conforming Hamiltonian path.
\end{lemma}

\begin{figure}
\centering
\noindent\includegraphics[scale=0.9]{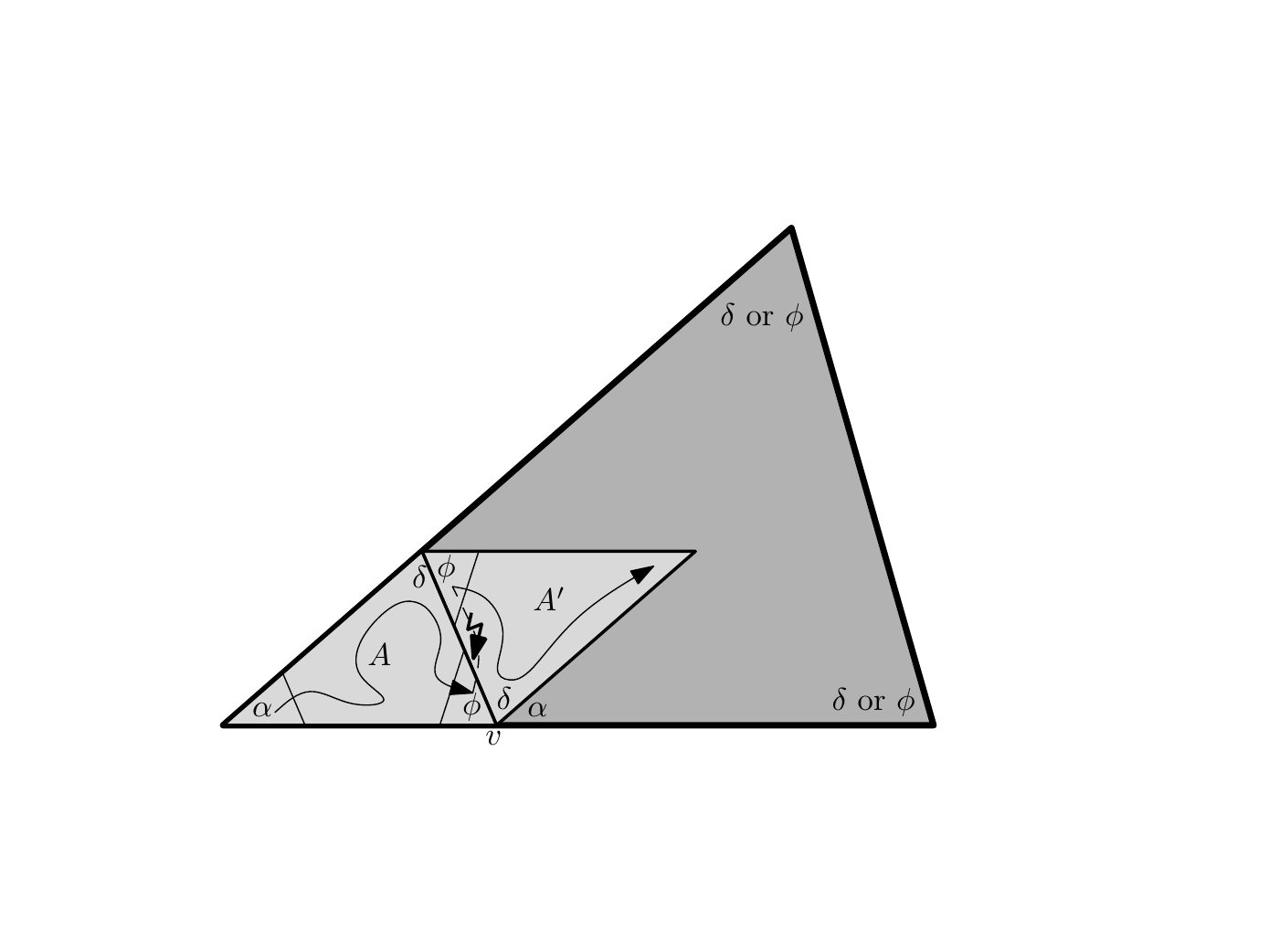}\hfill\includegraphics[scale=0.9]{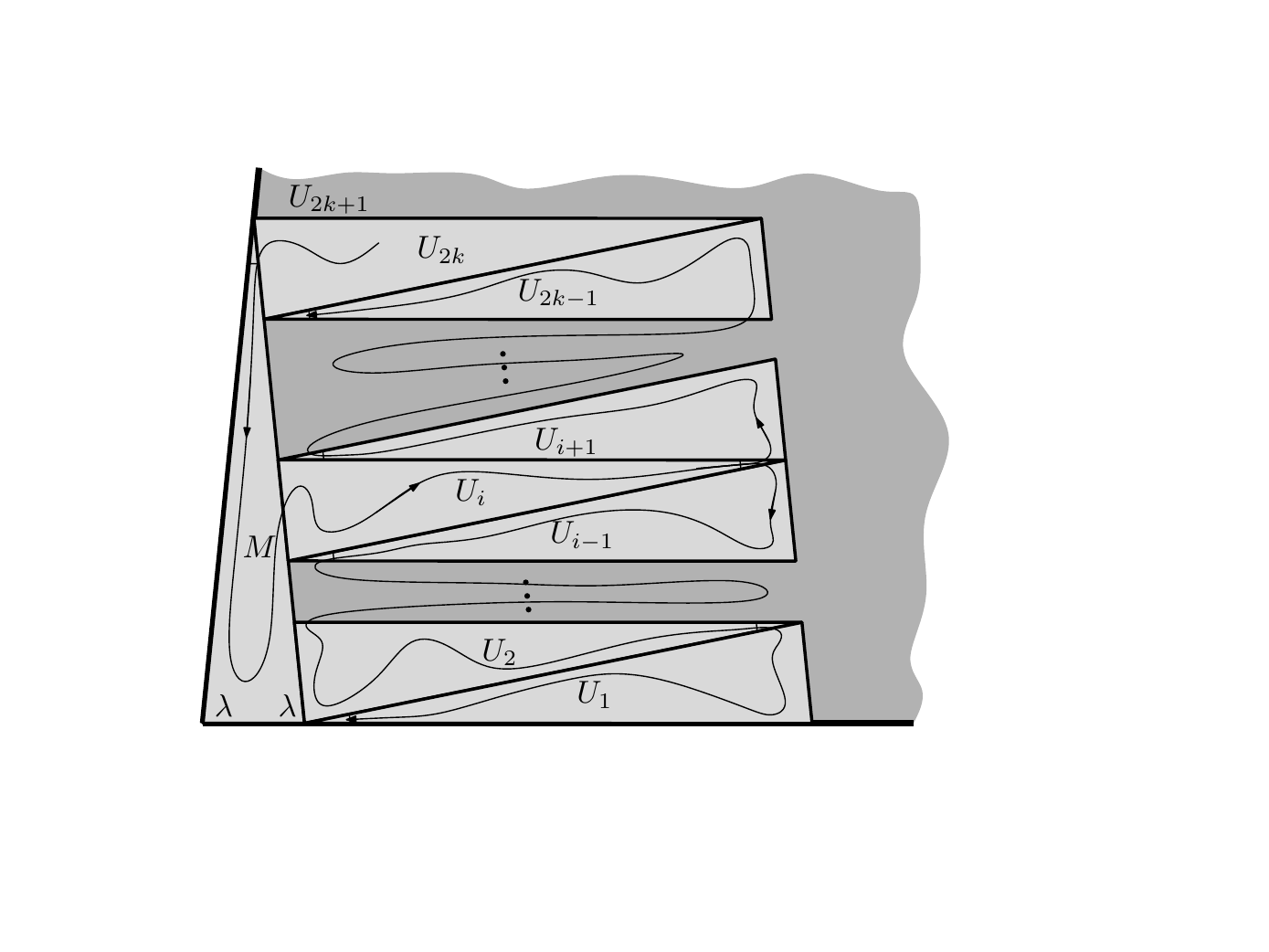}
\caption{Illustration of the proof of Lemma~\ref{lem:noconformingpath} for a scalene triangle (left) and an isosceles triangle (right).}
\label{fig:noconformingpath}
\end{figure}

\begin{proof}
Let $\alpha \leq \beta \leq \gamma$ be the angles of an acute triangle $T$.
Let the $r$-reptilings defining a two-level reptiling ${\mathcal T}_2$ of $T$ be obtained by applying reptilings recursively to a sufficiently deep level of recursion, such that $r > (2\sin\gamma / \sin\alpha)^2$. This ensures that no intermediate tile touches more than one master vertex, and the longest edge of any tile of ${\mathcal T}_2$ is less than half as long as the shortest edge of any intermediate tile.

We distinguish two cases: isosceles triangles, and scalene triangles.

We first discuss scalene triangles. For the sake of contradiction, suppose a conforming Hamiltonian path exists. The construction of our contradiction is illustrated in Figure~\ref{fig:noconformingpath}~(left). By Lemma~\ref{lem:uniquepath}, at most one vertex of a scalene acute triangle is a potential fan or cap. We denote the angle at this vertex by $\delta$, the smallest angle by $\alpha$, and the remaining angle by $\phi$. Thus, in the dual of any (single-level) reptiling, a Hamiltonian path, if one exists, has to start with the tile in the corner with angle $\alpha$ and end with the tile in the corner with angle $\phi$, or vice versa. Now consider the intermediate tile $A$ in the corner with angle $\alpha$. This tile shares its interior edge with a single other intermediate tile $A'$, which can only be placed in one way, namely with its $\phi$ angle adjacent to the $\delta$ angle of $A$ and with its $\delta$ angle adjacent to the $\phi$ angle of $A$ (otherwise the largest angle of $A'$ would meet the largest angle of $A$, and this would leave a gap with an angle smaller than $\alpha$ between $A'$ and the boundary of the master triangle, so the tiling could not be completed). Now the conforming Hamiltonian path through the tiling would have to end its traversal of $A$ with the atomic tile in the corner $v$ of size $\phi$, and then continue into $A'$. However, the only possible starting points for a Hamiltonian path through $A'$ are the tiles in its $\alpha$- and $\phi$-corners, but these tiles lie at the opposite ends of the edges of $A'$ that meet in $v$, and by our choice of $r$, these tiles are too small to be able to touch the tile in the $\phi$-corner of $A$. It follows that a conforming Hamiltonian path is not possible.

For the case of an isosceles triangle, let $\lambda$ and $\tau$ be the base and top angle of the triangle, respectively. By Theorem~\ref{thm:noisoscelesksplittingreptiling}, the triangle does not admit any fans. Caps at the top are not possible either, since the base of an isosceles triangle cannot be at least twice as long as a leg. Caps at the base corners, however, are possible, provided the length of a leg is $k$ times the length of the base for some natural number $k \geq 2$. This implies that the top angle is smaller than the base angle, and henceforth we will speak of small angles/corners (of which each tile has one) and large angles/corners (of which each tile has two).

Assume, for the sake of contradiction, that there exists a conforming Hamiltonian path in a two-level reptiling of an isosceles triangle. The construction of our contradiction is illustrated in Figure~\ref{fig:noconformingpath}~(right). Since there are no fans, for each corner there is a unique intermediate tile touching that corner. Let $M$ be an intermediate tile touching a corner that is neither first nor last on the Hamiltonian path---therefore that corner has to be a cap, and it is one of the larger corners.
Without loss of generality, assume $M$ is oriented such that its short edge (base) is horizontal, its small vertex (top) lies above it, and the master vertex is the left base corner of $M$. Let $U_2,U_4,U_6, \ldots ,U_{2k}$ be the adjacent intermediate tiles whose base (short) edges cover the right edge of $M$, in order from bottom to top (where $1/k = 2 \cos\lambda$), and let $U_1,U_3,U_5, \ldots ,U_{2k+1}$ be the intermediate tiles that have a single corner on the right edge of $M$, in order from bottom to top. Note that given the placement of $M$, the placement of $U_1, \ldots ,U_{2k}$ is fixed---only $U_{2k+1}$ could be a translate of $M$ or a translate of $U_{2k-1}$. 

Observe that any Hamiltonian path in any reptiling of any intermediate tile must either start or end with the tile in the small vertex. Without loss of generality, assume the traversal of $M$ starts at the top (otherwise we could complete the argument based on the reverse of the space-filling curve). That means that the intermediate tile visited immediately before $M$ must be $U_{2k}$. Therefore, immediately after $M$, we must visit one of the tiles $U_i$, for some $i \in \{2,4,6,...,2k-2\}$. Since we enter $U_i$ on the left side, that is, not at its small corner, we must end the traversal at its small corner, on the right. From there, the traversal can continue in either $U_{i-1}$ or $U_{i+1}$. Since we enter $U_{i-1}$ or $U_{i+1}$ at a large angle, the traversal must end at the small angle, on the left, touching $M$. From there we have no choice but to continue in $U_{i-2}$ or $U_{i+2}$, respectively. This inevitably leads to getting stuck in the left corner of either $U_{2k-1}$ (being unable to continue in $U_{2k}$ because it was already visited) or $U_1$ (being unable to continue because at its left corner, it is adjacent only to $U_2$, which is where we came from).
\end{proof}

Theorem~\ref{main3} now follows as a direct corollary of Lemmas \ref{lem:hamiltonianpath} and \ref{lem:noconformingpath}:
it is not possible to construct a face-continuous space-filling curve based on reptilings of an acute triangle.

\section{Topics for further research}\label{sec:furtherresearch}

Our original aim was to prove the existence or non-existence of face-continuous space-filling curves based on gentilings of acute triangles. While we were working on this problem, the questions for non-trivial reptilings and corner-splitting gentilings emerged as puzzles of independent entertainment value, and we have not completely solved them.

The first open question that remains is the following:

\begin{problem}
How can we fully classify all triangles that admit a non-trivial reptiling?
\end{problem}
\begin{facts}
We know that all right triangles admit non-trivial reptilings and we have established that all rational triangles admit non-trivial reptilings (Theorem~\ref{thm:nontrivialconstruction}). Theorem~\ref{thm:onlytrivial} states that irrational acute triangles and irrational oblique isosceles triangles do not admit non-trivial reptilings. This leaves the irrational scalene obtuse triangles undecided.

More precisely, the problem is still open for the irrational triangles with angles $\pi/m$, $k\pi/m$ and $(m-k-1)\pi/m$, for $k, m \in \mathbb{N}$, where $2 \leq k < m/2-1$ and there are non-zero natural numbers $\lambda, \mu, \nu$ such that $\lambda\sin(\pi/m) = \mu\sin(k\pi/m) + \nu\sin((k+1)\pi/m)$. Otherwise class (i) in the non-existence proof for non-trivial reptilings in Section~\ref{sec:irrationalistrivial} would still apply, or Observation~\ref{obs:stillcasei} or Observation~\ref{obs:case3general} applies and the analysis of class (ii) would still go through. In fact, an easy pigeon-hole argument based on Lemma~\ref{lem:euler}, similar to the arguments used in Section~\ref{sec:nosplitting}, shows that Observation~\ref{obs:case3} can also be established for obtuse triangles with an angle larger than $2\pi/3$, since then, no more than two such angles can meet in a point; therefore we may restrict $k$ further by $m/3-1 \leq k < m/2-1$.

Is this class of irrational triangles actually non-empty, and if so, does any such triangle admit a non-trivial reptiling? (If so, Theorem~\ref{thm:withouthalfverticesonlygrid} tells us that the tiling must include interior half vertices.) Or can we prove that none of these triangles admits a non-trivial reptiling, and hence, the right triangles and the rational triangles are the only triangles admitting non-trivial reptilings?
\end{facts}

\begin{table}
\centering\noindent
\begin{tabular}{llcc}
$k$ & conditions & reptilings & gentilings \\
\hline
2 & $\pi/6 < \alpha < \pi/4$; $\cos\alpha \in \mathbb{Q}$ & \textsc{exist} (smallest?) & \textsc{exist} \\
2 & $\pi/6 < \alpha < \pi/4$; $\cos\alpha \notin \mathbb{Q}$; $\cos^2\alpha \in \mathbb{Q}$ & ? & \textsc{exist} \\
2 & $\pi/6 < \alpha < \pi/4$; $\cos^2\alpha \notin \mathbb{Q}$; $2\cos\alpha - 1/(2\cos\alpha) \in \mathbb{Q}$ & ? & ? \\
2 & $\pi/6 < \alpha < \pi/4$; $\cos^2\alpha \notin \mathbb{Q}$; $2\cos\alpha - 1/(2\cos\alpha) \notin \mathbb{Q}$ & \textsc{impossible} & ? \\
3 & $\alpha \in \{\frac2{13}\pi, \frac2{15}\pi\}$ & \textsc{impossible} & ? \\
3 & $\pi/8 < \alpha < \pi/6$; $\alpha \notin \{\frac2{13}\pi, \frac2{15}\pi\}$ & \textsc{impossible} & \textsc{impossible} \\
$\geq 4$ & $\pi/(2k+2) < \alpha < \pi/(2k)$ & \textsc{impossible} & \textsc{impossible} \\
\end{tabular}
\caption{Our current state of knowledge about the existence of $k$-splitting gentilings of acute triangles with angles $\alpha$, $k\alpha$ and $\pi-(k+1)\alpha$. The lower and upper bounds on $\alpha$ in the table only serve to ensure that the triangle is indeed acute.}
\label{tab:splittingsummary}
\end{table}

Table~\ref{tab:splittingsummary} summarizes our current state of knowledge about the existence of $k$-splitting gentilings of acute triangles. We are left with the following open problems with regard to the existence of $k$-splitting \emph{reptilings} of oblique triangles:

\begin{problem}\label{pro:smallest2splitting}
Is there an oblique triangle that admits a 2-splitting $r$-reptiling with $r < 169$?
\end{problem}
\begin{problem}
Is there an oblique triangle with angles $\alpha$, $2\alpha$ and $\pi-3\alpha$ with $\cos\alpha \notin \mathbb{Q}$, that admits a 2-splitting reptiling?
\end{problem}
\begin{facts}
If we drop the requirements $\cos\alpha \notin \mathbb{Q}$ and $r < 169$, a construction is given in Section~\ref{sec:splittingconstruction}. If we drop the requirement that the triangle is not a right triangle, then
all right triangles, regardless of $\cos\alpha$, are solutions allowing a 2-splitting 4-reptiling.
\end{facts}

For gentilings, we can easily extend our negative results to obtuse triangles with an angle larger than $2\pi/3$:

\begin{theorem}\label{thm:noveryobtuse3splitting}
No triangle $T$ with an angle greater than $2\pi/3$ admits a $k$-splitting gentiling for $k \geq 3$.
\end{theorem}
\begin{proof}
Let the angles of $T$ be $\alpha \leq \beta < \gamma$, where $\gamma > 2\pi/3$.

By Lemma~\ref{lem:euler} we have $2r = 4f + 2h + 2$. This means that the $\gamma$ angles of the tiles in a tiling ${\cal T}$ of $T$ are exactly accounted for if each full vertex of $G_{\cal T}$ fits two, each half vertex fits one, and there is one at a master vertex. Since $3\gamma > 2\pi$ and $2\gamma > \pi$, this is also the most that will fit at each of these vertices. Hence, the full vertices, half vertices, and master vertices must fit exactly two $\gamma$ angles per vertex, one $\gamma$ angle per vertex, and one $\gamma$ angle in total, respectively. In particular, the master vertex with the $\gamma$ angle cannot be split, so we must have $\beta = k\alpha$.

Now, by Lemma~\ref{lem:euler}, the $\beta$ angles of the tiles in a tiling are exactly accounted for if each full vertex of $G_{\cal T}$ fits two, each half vertex fits one, and there is one at a master vertex. Since $2\gamma + 3\beta = 2\gamma + 2\beta + k\alpha > 2\pi$ and $\gamma + 2\beta = \gamma + \beta + k\alpha > \pi$, this is also the most that will fit at each of these vertices. Hence, the $\beta$ angle cannot be split either. It follows that a $k$-splitting gentiling is not possible.
\end{proof}

This leaves several open problems with regard to the existence of $k$-splitting \emph{gentilings} of triangles, including:

\begin{problem}
Is there an oblique triangle with angles $\alpha$, $2\alpha$ and $\pi-3\alpha$ with $\cos^2\alpha \notin \mathbb{Q}$ that admits a 2-splitting gentiling?
\end{problem}
\begin{problem}
\begin{compactitem}
\item[(i)] Does the triangle with angles $\frac2{13}\pi$, $\frac5{13}\pi$ and $\frac6{13}\pi$ admit a 3-splitting gentiling?
\item[(ii)] Does the triangle with angles $\frac2{15}\pi$, $\frac6{15}\pi$ and $\frac7{15}\pi$ admit a 3-splitting gentiling?
\end{compactitem}
\end{problem}


\begin{problem}
What is the largest $k$ for which there exists a $k$-splitting gentiling?
\end{problem}
\begin{facts}
The isosceles triangle with $\alpha = \frac16\pi$ admits a 4-splitting 5-gentiling (Figure~\ref{fig:facecontinuouscurves}(d)). From Theorems \ref{thm:splittingangles} and~\ref{thm:noveryobtuse3splitting} we get that the triangle that maximizes $k$ must be non-acute with largest angle at most $2\pi/3$.
\end{facts}

Finally consider our original goal. The proof of Section~\ref{sec:sfc} that no face-continuous space-filling curves exist on the basis of reptilings of an acute triangle, crucially exploits the fact that the triangle is acute and that tiles have the same size and that this makes it impossible to construct useful caps. The following questions are still completely open:

\begin{problem}
Is there a face-continuous space-filling curve based on a reptiling of an \emph{obtuse} triangle?
\end{problem}
Note that with respect to the original motivation, obtuse triangles are probably not interesting,
as they would not be better-shaped than the isosceles right triangles underlying the Sierpi\'nski curve. Therefore, unless an application appears which would really require oblique mesh elements, the above problem may mostly be interesting for entertainment.

\begin{figure}
\centering
\includegraphics[width=\hsize]{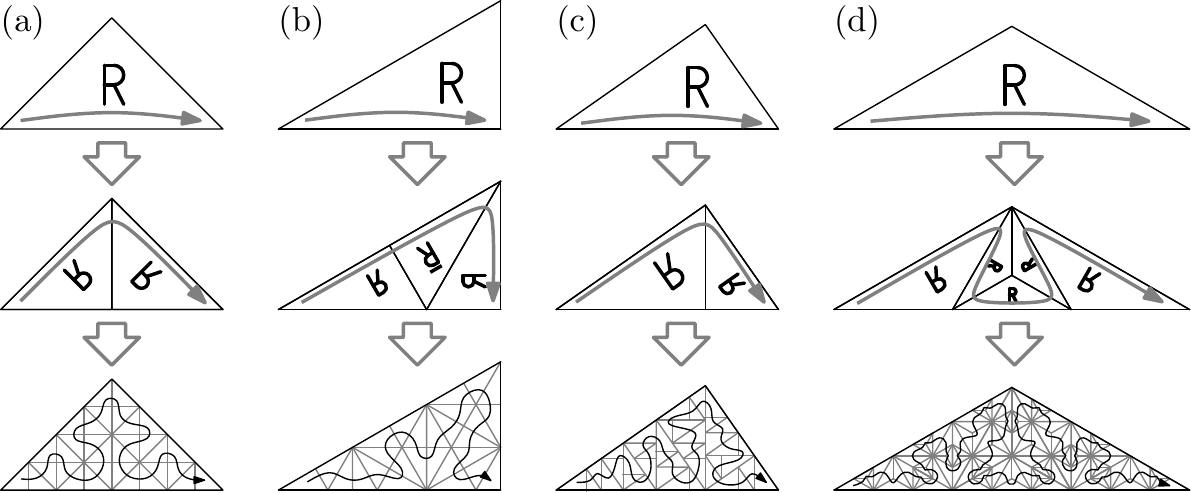}
\caption{Four face-continuous space-filling curves based on gentilings of non-acute triangles. The top row shows, for each curve, the master tile; the next row shows how this tile is subdivided into smaller tiles, in what order the tiles are traversed by the space-filling curve, and what transformations map the master tile to each tile (an overscore on the \textsf{R} signifies reversal of the order in which a tile is traversed). The bottom row shows the tiling that results from applying the recursive tiling rules until each tile has at most 1/25 of the area of the master tile. In all figures, the curve with the arrow head shows in what order the tiles are traversed. (a) The Sierpi\'nski curve, based on a reptiling of the isosceles right triangle.\quad (b) A curve based on a reptiling of the right triangle with angles $\pi/6$, $\pi/3$ and $\pi/2$. To the best of our knowledge, the curve is known but unnamed.\quad (c) An example of a P\'olya curve based on a gentiling of a right triangle. Such curves can be constructed for any right triangle.\quad (d) A curve based on a gentiling of the isosceles triangle with angles $\pi/6$ and $2\pi/3$. Note that the same curve can be constructed from two copies of curve (b).}
\label{fig:facecontinuouscurves}
\end{figure}

\begin{problem}
Is there a face-continuous space-filling curve based on a \emph{gentiling} of an acute triangle?
\end{problem}
\begin{facts}
Note that face-continuous space-filling curves can be constructed based on:\begin{compactitem}
\item a 2-splitting 2-reptiling of the isosceles right triangle (the Sierpi\'nski curve~\cite{sagan}, Figure~\ref{fig:facecontinuouscurves}(a));
\item a 2-splitting 3-reptiling of the right triangle with angle $\frac16\pi$ (Figure~\ref{fig:facecontinuouscurves}(b)).
\item a 2-splitting 2-gentiling of any right triangle (P\'olya curves~\cite{sagan}, Figure~\ref{fig:facecontinuouscurves}(c));
\item a 4-splitting 5-gentiling of the isosceles obtuse triangle with base angle $\frac16\pi$ (Figure~\ref{fig:facecontinuouscurves}(d));
\end{compactitem}
\end{facts}

\paragraph{Acknowledgements}
We thank Dirk Gerrits for his help in obtaining our first proofs of Theorem~\ref{thm:no3splittingreptiling}.

\bibliography{bibliography}{}
\bibliographystyle{amsplain}

\end{document}